\documentclass[a4paper]{article}

\usepackage[latin1]{inputenc} 
\usepackage[T1]{fontenc} 
\usepackage[left=2.2cm,right=2.2cm,top=2cm,bottom=2.2cm]{geometry}
\usepackage{amsmath,amsthm,amssymb} \let\amsmathbb\mathbb 
\usepackage[dvipsnames,table]{xcolor} 
\usepackage{graphicx}
\usepackage{mathtools} 
\usepackage{stmaryrd} \SetSymbolFont{stmry}{bold}{U}{stmry}{m}{n} 
\usepackage{esint} 
\usepackage[english]{babel} 
\usepackage{float} 
\usepackage{enumitem}
\usepackage{multirow,multicol,paracol,booktabs} 
\usepackage{lipsum} 
\usepackage[normalem]{ulem} 
\usepackage{tabularx}\newcolumntype{C}{>{\centering\arraybackslash}X} 
\usepackage{xspace}
\usepackage{clipboard} 
\usepackage[overload]{empheq} 
\usepackage{mathrsfs} 
\usepackage{fourier} 
\usepackage[cal=cm,scr=boondox]{mathalpha} 
\usepackage[backend=bibtex,giveninits=true,style=trad-alpha,url=false,doi=false,isbn=false,sorting=nty,style=alphabetic,maxnames=50]{biblatex}
\usepackage[hidelinks,colorlinks,allcolors=.,urlcolor=teal]{hyperref} 
\usepackage[noabbrev]{cleveref} 

\crefdefaultlabelformat{#2#1#3}
\crefname{equation}{}{} \creflabelformat{equation}{(#2#1#3)}
\crefname{section}{section}{sections} \Crefname{section}{Section}{Sections}
\crefname{subsection}{}{} \Crefname{subsection}{Section}{Sections}
\crefname{enumi}{}{} \Crefname{enumi}{Point}{Points}
\creflabelformat{appx}{#2Appendix#3}\Crefname{appx}{}{}

\allowdisplaybreaks 

\AtEveryBibitem{\clearfield{location}}
\AtEveryBibitem{\clearfield{eprint}}
\AtEveryBibitem{\clearname{editor}}
\AtEndDocument{
	\printbibliography
}

\setlist[itemize,1]{label=$-$}
\setlist[itemize,2]{label=$-$}
\setlist[itemize,3]{label=$-$}

\makeatletter
\NewCommandCopy\latexparagraph\paragraph
\RenewDocumentCommand{\paragraph}{sO{#3}m}{%
  \IfBooleanTF{#1}
    {\latexparagraph*{\maybe@addperiod{#3}}}
    {\latexparagraph[#2]{\maybe@addperiod{#3}}}%
}
\newcommand{\maybe@addperiod}[1]{%
  #1\@addpunct{.}%
}
\makeatother

\makeatletter
\newcommand{\ncite}{\@ifstar{\@ncite}{\@@ncite}}
\newcommand{\@ncite}[2][]{\citeauthor*{#2} \ifx#1\empty\cite{#2}\else\cite[#1]{#2}\fi}
\newcommand{\@@ncite}[2][]{\citeauthor{#2} \ifx#1\empty\cite{#2}\else\cite[#1]{#2}\fi}
\makeatother

\AtBeginDocument{%
	\let\mathbb\relax%
	\newcommand{\mathbb}[1]{\amsmathbb{#1}}%
}

\theoremstyle{definition}

\newtheorem{definition}{Definition} \crefname{definition}{definition}{definitions} \Crefname{definition}{Definition}{Definitions}
\newtheorem*{definition*}{Definition}

 \crefname{conjecture}{conjecture}{conjectures} \Crefname{conjecture}{Conjecture}{Conjectures}
\newtheorem*{conjecture*}{Conjecture}

 \crefname{hyp}{}{} \Crefname{hyp}{}{} \creflabelformat{hyp}{\textbf{[#2A#1#3]}}

\theoremstyle{theorem}

\newtheorem{theorem}{Theorem} \crefname{theorem}{theorem}{Theorem} \Crefname{theorem}{Theorem}{Theorems}
\newtheorem*{theorem*}{Theorem}

\newtheorem{lem}[definition]{Lemma} \crefname{lem}{lemma}{lemmata} \Crefname{lem}{Lemma}{Lemmata}
\newtheorem*{lem*}{Lemma}

\newtheorem{prop}[definition]{Proposition} \crefname{prop}{proposition}{propositions} \Crefname{prop}{Proposition}{Propositions}
\newtheorem*{prop*}{Proposition}

 \crefname{cor}{corollary}{corollaries} \Crefname{cor}{Corollary}{Corollaries}
\newtheorem*{cor*}{Corollary}

\theoremstyle{remark}

 \crefname{example}{example}{examples} \crefname{example}{Example}{Examples}

\newtheorem{rem}[definition]{Remark} \crefname{rem}{remark}{remarks} \Crefname{rem}{Remark}{Remarks} 
\newtheorem*{rem*}{Remark}

\newcommand{\highlight}[2]{\begin{center}\small\begin{minipage}{0.9\textwidth}\begin{center}\textbf{#1}\end{center}\vspace*{-7pt}#2\end{minipage}\end{center}}

\let\div\undefined\DeclareMathOperator*{\div}{div\,} 

\DeclareMathOperator*{\supp}{supp\,}
\DeclareMathOperator*{\diam}{diam\,}

\newcommand{\ind}{{\textrm{1\hspace{-0.6ex}I}}}
\newcommand{\st}{\ \middle| \ } 
\newcommand{\bdot}{{\boldsymbol{\cdot}}} 

\newcommand{\resmes}{\mathbin{\vrule height 1.6ex depth 0pt width 0.18ex\vrule height 0.15ex depth 0pt width 0.7ex}} 

\makeatletter
\newcommand{\stephighlight}[1]{\medskip\noindent\textbf{#1\@addpunct{.}}\quad} 
\makeatother 

\newcommand{\Sp}{\mathbb{R}}
\DeclareMathOperator{\T}{T} 
\DeclareMathOperator{\Psp}{{\mathscr{P}}} 
\DeclareMathOperator{\Tan}{\mathbf{Tan}} 
\DeclareMathOperator{\rTan}{Tan} 
\DeclareMathOperator{\Lipsp}{Lip} 

\newcommand{\bary}[2]{{\text{Bary}_{#2}\left(#1\right)}} 
\newcommand{\Lip}[1]{\text{Lip}(#1)}

\title{Characterization of measures on the real line \\ that are critically unstable under small shifts}
\author{Averil Aussedat\footnote{Dipartimento di Matematica, Universit\`a di Pisa, largo Pontecorvo 5, 56127 Pisa, Italy. \texttt{averil.aussedat@dm.unipi.it}.}}
\date{}

\begin{document}

\maketitle

\highlight{Abstract}{We study the perturbation of a measure $\mu \in \mathscr{P}(\mathbb{R})$ consisting in superposing two copies of $\mu$, each slightly shifted by a small distance $\pm h$. The difference between $\mu$ and its perturbation is measured with a Wasserstein distance. For any $\mu$, this distance is bounded from above by $h$. We show that measures for which this critical rate is achieved when $h$ goes to 0 are characterized as the ones giving most of their mass to some particular porous sets. This is used to identify which measures $\mu$ on the real line have a 2-Wasserstein tangent cone equal to the set of directions inducing curves with maximal initial speed.}

\medskip

\medskip


\section*{Introduction}

\renewcommand{\Sp}{\mathbb{R}}

The set $\Psp(\Sp)$ of probability measures can be endowed with the family of Wasserstein distances $W_p(\cdot,\cdot)$, measuring the minimal $L^p-$travel distance for an exponent $p \in [1,\infty)$. All these distances coincide in some particular cases: for instance, if $\mu = \delta_0$ and $\mu_h \coloneqq \frac{1}{2} \left[\delta_{-h} + \delta_h\right]$, then $W_p(\mu,\mu_h) = h$ for any $h \geqslant 0$ and $p \in [1,\infty)$. This article focuses on the set of measures satisfying the same inequality at the limit when $h$ goes to zero. Precisely, given $\mu \in \Psp(\Sp)$, we consider the superposition of its left and right shift
\begin{align*}
	\mu_h \coloneqq \frac{1}{2} \left[(id - h)_{\#} \mu + (id + h)_{\#} \mu\right],
\end{align*}
and ask whether it holds that 
\begin{align}\label{intro:limsup1}
	\limsup_{h \searrow 0} \frac{W_p(\mu,\mu_h)}{h} = 1.
\end{align}
The inequality $\leqslant$ always holds, and for absolutely continuous measures, the limit sup in is equal to 0 instead of 1. Hence the condition~\Cref{intro:limsup1} is a critical behavior, and we expect any $\mu$ satisfying it to be very concentrated. It turns out that the set of such $\mu$ is independent of $p \in [1,\infty)$ (see \Cref{rem:limindepp} below). The first result of the paper states that \Cref{intro:limsup1} is equivalent to the fact that $\mu$ gives almost all its mass to an element of a very restrictive class $\mathscr{A}$ of porous set. 

\begin{definition}\label{intro:def:porous}
	A measurable set $A \subset \mathbb{R}$ belongs to $\mathscr{A}$ if for some $(s_n)_{n \in \mathbb{N}} \subset (0,1)$ going to 0 when $n \to \infty$, there holds
	\begin{align}\label{intro:def:porous:statement}
		\lim_{n \to \infty} \sup_{x \in A} \inf \left\{ \tau \in (0,1) \st \mathscr{B}(x,s_n) \cap A \subset \overline{\mathscr{B}}(x,\tau s_n) \right\} = 0.
	\end{align}
\end{definition}
In words, in any ball of radius $s_n$ around $x \in A$, the set $A$ is contained in the smaller ball of radius $\tau s_n$ for some $\tau$ that goes to 0 with $s_n$.
Sets in $\mathscr{A}$ are porous in the sense that each ball $\mathscr{B}(x,s_n)$ centred at $x \in A$ contains a ball of diameter larger than $(1-\tau) s_n$ that does not intersect $A$ (see for instance~\cite{meraPorositysigmaporosityMeasures2003}), but \Cref{intro:def:porous} is stronger in general: the same sequence of scales $(s_n)_n$ is shared by all points $x \in A$, and $\tau$ is as small as desired. 

\begin{theorem}\label{res:charLS1}
	Let $\mu \in \Psp(\Sp)$. The following conditions are equivalent:
	\begin{enumerate}[label={(\alph*)}]
	\item\label{item:charLS1:limit} the equality \Cref{intro:limsup1} holds for some (thus any) $p \in [1,\infty)$;
	\item\label{item:charLS1:concen} for any $\varepsilon > 0$, there exists $A \in \mathscr{A}$ such that $\mu(A) \geqslant 1 - \varepsilon$.
	\end{enumerate}
\end{theorem}

The motivation of \Cref{res:charLS1} is to solve a pending question on the geometry of the 2-Wasserstein space. For $p = 2$, one can use the existing theory of Alexandrov spaces \cite{alexanderAlexandrovGeometryFoundations2023} to construct a geometric tangent cone attached to each measure \cite{ambrosioGradientFlows2005,gigliGeometrySpaceProbability2008}. The elements of this cone can be represented by measure-valued applications in $L^2_{\mu}(\mathbb{R}^d;\Psp_2(\mathbb{R}^d))$. The subset of elements induced by a map, i.e. of the form $x \mapsto \delta_{f(x)}$ for some $f \in L^2_{\mu}(\mathbb{R}^d;\mathbb{R}^d)$, has proved to be useful in studying differentiability of functions of measures \cite{bressanShiftdifferentiabilityFlowGenerated1997,ottoGeometryDissipativeEvolution2001,ambrosioGradientFlows2005,gangboHamiltonJacobiEquationsWasserstein2008,carmonaProbabilisticTheoryMean2018,gangboDifferentiabilityWassersteinSpace2019}. In recent years however, the general case of elements that could take values outside of Dirac masses has gained attention, both to generalize continuity equations with ``measure fields'' and unify discrete and continuous treatments of numerical schemes (formulated with trajectories that ``split mass'' at the discrete level)  \cite{piccoliMeasureDifferentialEquations2019,camilliSuperpositionPrincipleSchemes2021,cavagnariLagrangianApproachTotally2023,cavagnariDissipativeProbabilityVector2023}, or as a theoretically appealing definition for optimization and differentiability over measures \cite{lanzettiFirstorderConditionsOptimization2022,bertucciStochasticOptimalTransport2024,schichlNonlinearDegenerateParabolic2025,bertucciTangentSpaceWasserstein2025,ceccherinisilbersteinViscositySolutionsHamiltonJacobi2026}. 

\medskip 
In this transition from map-induced to general elements of the tangent cone, some properties are lost. The one that motivates \Cref{res:charLS1} is the following: if $f \in L^2_{\mu}(\mathbb{R}^d;\mathbb{R}^d)$ is a vector field, then the fact that $x \mapsto \delta_{f(x)}$ belongs to the tangent cone is equivalent to the fact that $W_2(\mu,(id + h f)_{\#} \mu)/h = \|f\|_{L^2_{\mu}} + O(h)$ \cite[Proposition~5.3.8]{aussedatOptimalControlProblems2025}. This equivalence replaces the definition of the tangent cone, which goes by approximation by geodesics in a strong topology, by a simple computation of the initial speed of the curve $h \mapsto (id + h f)_{\#} \mu$. Unfortunately, tt does not hold in the general case. The second result of this paper is to characterize the set of measures $\mu$ for which 
\begin{align}\label{intro:equivTan}
	\xi \in \Tan_{\mu}
	\qquad \iff \qquad
	\lim_{h \searrow 0} \frac{W_2(\mu,(\pi_x + h \pi_v)_{\#} \xi)}{h} = \|\xi\|_{\mu}.
\end{align}

\begin{theorem}\label{res:answerQ1}
	Let $\mu \in \Psp_2(\Sp)$. The following conditions are equivalent:
	\begin{enumerate}
	\item the equivalence \Cref{intro:equivTan} holds; 
	\item the diffuse part of $\mu$ satisfies $\mu^{d}(A) = 0$ for any $A \in \mathscr{A}$.
	\end{enumerate}
\end{theorem}

Examples for which \Cref{intro:equivTan} fails can be found in \cite{aussedatStructureGeometricTangent2025} or \Cref{sec:LS0}. The case of dimension $d > 1$ is open. 

\medskip
In the last part of the paper, we consider the related problem of measures for which the limit sup in~\Cref{intro:limsup1} is equal to 0, instead of 1. This is still independent of $p \in [1,\infty)$. The motivation comes from the fact that for a map $g \in L^2_{\mu}(\mathbb{R}^d;\mathbb{R}^d)$, it is equivalent that $W_2(\mu,(id+h g)_{\#} \mu) = o(h)$, and that $\left<g,f\right>_{L^2_{\mu}} = 0$ for any $f \in L^2_{\mu}$ inducing a tangent element. Such $g$ are actually the $L^2_{\mu}$ vector fields for which $\div(g \mu) = 0$ in the sense of distributions, and the fact that they let $\mu$ ``infinitesimally invariant'' extends the theorem of Liouville on incompressible flows being driven by divergence-free vector fields. Once again, this does not generalize to measure-valued fields, hence the need for some regularity condition on $\mu$ under which the equivalence would stand.

\medskip 
We do not have a fully satisfactory characterization of the set of measures $\mu$ such that $W_1(\mu,\mu_h) = o(h)$, besides some regularity of the distribution function of $\mu$. However, we can exhibit such a $\mu$ that is singular with respect to the Lebesgue measure. This seems to us a curiosity in itself: indeed, from the definition, it holds that for any Lipschitz function $u : \mathbb{R} \to \mathbb{R}$, 
\begin{align*}
	\lim_{h \searrow 0} \left|\int_{x \in \Sp} \frac{\frac{1}{2} \left[u(x-h) + u(x+h)\right] - u(x)}{h} d\mu\right|
	\leqslant \Lip{u} \frac{W_1(\mu,\mu_h)}{h}
	\underset{h \searrow 0}{\longrightarrow} 0.
\end{align*}
If the absolute values were inside the first integral, we would conclude to a kind of differentiability of any Lipschitz function $\mu-$almost everywhere. However, the measures for which Lipschitz functions are differentiable a.e. are exactly the absolutely continuous measures \cite{zahorskiLensemblePointsNonderivabilite1946,fowlerSimpleProofZahorskis2008}. Here, we do not exactly consider ``differentiability'', but an averaged version on both the directions and on the points; still, it was not so obvious that some singular measures would appear.

\medskip
The paper is organized as follows. \Cref{sec:prelim} precises the notations, definitions and facts in use in the sequel. \Cref{res:charLS1} is proved in \Cref{sec:LS1}. \Cref{res:answerQ1} is proved in \Cref{sec:linkTan}. \Cref{sec:LS0} concludes with some explicit computations around the problem mentioned in the end of the introduction.

\tableofcontents

\section{Preliminaries}\label{sec:prelim}

The open ball of radius $r \geqslant 0$ around a point $x \in \Sp$ is denoted $\mathscr{B}(x,r)$. The closure of a set $A \subset \mathbb{R}$ is denoted $\overline{A}$, and its complement $A^c \coloneqq \Sp \setminus A$. A vanishing sequence $(h_n)_{n \in \mathbb{N}} \subset (0,\infty)$ is a nonincreasing sequence that converges to 0. If $\mu \in \Psp(X)$ is a Borel probability measure over a Polish space $X$, and $f : X \to Y$ is a measurable map between Polish spaces, then $f_{\#} \mu$ is the probability measure on $Y$ determined by $(f_{\#} \mu)(A) = \mu(f^{-1}(A))$ for any measurable $A \subset Y$. The set $\T\Sp \coloneqq \left\{ (x,v) \st x \in \Sp, v \in \T_x\Sp \right\}$ is isometric to $\mathbb{R}^2$. 

The following material is classical, and can be found in detailed form in \cite{santambrogioOptimalTransportApplied2015} (in particular Proposition~2.17 for the monotone plan). For the dual formulation, we adopt the sign convention of \cite[Theorem 5.10]{villaniOptimalTransport2009}.

\paragraph{Wasserstein distances}

When $\mu$ is a measure on a set $X \times Y$, we use the notation $\mu = \mu(dx,dy)$ to give names to the variables of $\mu$, and later use them in the canonical projections $\pi_x, \pi_y$. The set of measures $\xi = \xi(dx,dv) \in \Psp_2(\T\Sp)$ that satisfy $\pi_{x \#} \xi = \mu$ is denoted $\Psp_2(\T\Sp)_{\mu}$. Each element of $\Psp_2(\T\Sp)_{\mu}$ induced a curve of measures denoted
\begin{align*}
	h \mapsto \exp_{\mu}(h \cdot \xi) \coloneqq (\pi_x + h \pi_v)_{\#} \xi. 
\end{align*}
The set of transport plans between $\mu \in \Psp(X)$ and $\nu \in \Psp(Y)$ is denoted
\begin{align*}
	\Gamma(\mu,\nu) \coloneqq \left\{ \ \eta = \eta(dx,dy) \in \Psp(X \times Y) \st \pi_{x \#} \eta = \mu, \text{ and } \pi_{y \#} \eta = \nu \ \right\}.
\end{align*}
For $p \in [1,\infty)$, let $\Psp_p(\Sp)$ be the subset of $\mu \in \Psp(\Sp)$ such that $\int_{x \in \Sp} |x|^p d\mu < \infty$.
The $p-$Wasserstein distance $W_p(\mu,\nu)$ between $\mu,\nu \in \Psp_p(\Sp)$ is defined as
\begin{align*}
	W_p^p(\mu,\nu) \coloneqq \inf_{\eta \in \Gamma(\mu,\nu)} \int_{(x,y) \in \Sp^2} |x - y|^p d\eta(x,y).
\end{align*}
In general, minimizers exists but are not unique, and we denote by $\Gamma_{o,p}(\mu,\nu)$ the set of minimizers (dropping the $p$ when it is clear from the context).
In the sequel, we might take Wasserstein distances between measures $\mu,\nu$ with $0 < \mu(\Sp) = \nu(\Sp) \neq 1$; in this case, we adopt the convention 
\begin{align*}
	W_p(\mu,\nu) \coloneqq \mu(\Sp) \times W_p\left(\frac{\mu}{\mu(\Sp)}, \frac{\nu}{\nu(\Sp)}\right).
\end{align*}
Similarly, we extend $W_p(\cdot,\cdot)$ to $\Psp(\Sp)^2$ by $\infty$ on pairs $(\mu,\nu)$ such that no transport plan yields a finite cost. Since we consider mainly the distance between $\mu$ and the shifted average $\mu_h$, it always holds that $W_p(\mu,\mu_h) \leqslant h < \infty$, regardless of the moments of $\mu$.

\paragraph{Monotone plan}

In dimension one, an explicit minimizer is known, which is in addition unique if $p > 1$. This plan is called the ``monotone plan'', since it is constructed by gathering all the mass that $\mu$ puts on $(-\infty,x]$, and spread it monotonically on $(-\infty,y]$ for some $y$ such that $\nu((-\infty,y)) \leqslant \mu((-\infty,x]) \leqslant \nu((-\infty,y])$. In rigorous terms, define the distribution function of $\mu \in \Psp(\Sp)$ as
\begin{align*}
	F_{\mu} : \Sp \to [0,1], \qquad
	F_{\mu}(x) \coloneqq \mu((-\infty,x]).
\end{align*}
The function $F_{\mu}$ is nondecreasing, and admits a pseudo-inverse $F_{\mu}^{[-1]} : [0,1] \to \Sp \cup \{-\infty\}$ characterized by $F_{\mu}(x) \geqslant r$ iff $x \geqslant F_{\mu}^{[-1]}(r)$. Explicitly, there holds $F_{\mu}^{[-1]}(r) \coloneqq \inf \left\{ x \in \Sp \st F_{\mu}(x) \geqslant r \right\}$. The monotone plan between $\mu$ and $\nu$ is defined as $\eta \coloneqq (F_{\mu}^{[-1]},F_{\nu}^{[-1]})_{\#} \mathcal{L}_{[0,1]}$. By optimality, it holds for any $p \in [1,\infty)$ that 
\begin{align*}
	W_p^p(\mu,\nu) = \int_{r \in [0,1]} \left|F_{\mu}^{[-1]}(r) - F_{\nu}^{[-1]}(r)\right|^p dr. 
\end{align*}
In the case $p = 1$, the expression simplifies further in $W_1(\mu,\nu) = \int_{x \in \Sp} \left|F_{\mu}(x) - F_{\nu}(x)\right| dx$. Here note that the integral is taken with respect to the Lebesgue measure. 

\paragraph{Dual formulation}

Let $c : \Sp^2 \to \mathbb{R}^+$ be any cost of the family $c(x,y) = |x - y|^p$. 
We follow \cite{villaniOptimalTransport2009} in defining the c--transform of a function $\varphi : \Sp \to \mathbb{R}$ as 
\begin{align*}
	\varphi^c : \Sp \to \mathbb{R} \cup \{-\infty\}, 
	\qquad
	\varphi^c(y) \coloneqq \inf_{x \in \Sp} \varphi(x) + c(x,y).
\end{align*} 
The $p-$Wasserstein distance admits a dual reformulation as
\begin{align*}
	W_p^p(\mu,\nu) = \sup_{\varphi \in L^1_{\mu}(\Sp;\mathbb{R})} \int_{y \in \Sp} \varphi^c(y) d\nu - \int_{x \in \Sp} \varphi(x) d\mu.
\end{align*}
Since the cost inside the supremum cannot decrease if $\varphi$ is replaced by $(\varphi^c)_c \coloneqq \sup_{y \in \Sp} c(x,y) - \varphi^c(y)$, one can assume that $\varphi$ is of this form. In particular, if $c(x,y) = |x-y|$, one checks that taking twice c--transforms yields that $\varphi$ is 1-Lipschitz, and $\varphi^c = \varphi$. Hence, denoting $\Lipsp_1$ the set of 1-Lipschitz functions, there holds
\begin{align}\label{dual:p=1}
	W_1(\mu,\nu) = \sup_{\varphi \in \Lipsp_1} \int_{y \in \Sp} \varphi(y) d\nu - \int_{x \in \Sp} \varphi(x) d\mu.
\end{align}

\paragraph{Tangent cones in the case $p=2$}

The regular tangent cone $\rTan_{\mu} \coloneqq \overline{\left\{ (id,\nabla \varphi)_{\#} \mu \st \varphi \in \mathcal{C}^{\infty}_c(\Sp;\mathbb{R}) \right\}}^{L^2_{\mu}}$ identifies as a subset of the geometric tangent cone, defined as
\begin{align*}
	\Tan_{\mu} \coloneqq \overline{\left\{ \, \lambda \cdot \xi \st (\pi_x,\pi_x + \pi_v)_{\#} \xi \in \Gamma_o(\mu,(\pi_x+\pi_v)_{\#} \xi), \ \text{and} \  \lambda \geqslant  0 \, \right\}}^{W_{\mu}} \subset \Psp_2(\T\Sp)_{\mu}.
\end{align*}
Here, $\lambda \cdot \xi \coloneqq (\pi_x, \lambda \pi_v)_{\#} \xi$ is the scalar multiplication of a measure field $\xi$, and the distance $W_{\mu}$ generalizes the $L^2_{\mu}$ distance as
\begin{align*}
	W_{\mu}^2(\xi,\zeta) \coloneqq \inf_{\eta \in \Gamma_{\mu}(\xi,\zeta)} \int_{(x,v,w)} |v - w|^2 d\eta, 
	\qquad \text{where} \qquad
	\Gamma_{\mu}(\xi,\zeta) \coloneqq \left\{ \eta\in \Psp_2\left(\bigcup_{x \in \Sp} \{x\} \times \T_x\Sp^2\right) \st \begin{matrix} (\pi_x,\pi_v)_{\#} \eta = \xi, \text{ and} \\ (\pi_x,\pi_w)_{\#} \eta = \zeta \end{matrix} \right\}.
\end{align*} 
We shorten $W_{\mu}(\xi,(id,0)_{\#} \mu)$ as $\|\xi\|_{\mu} = \sqrt{\int_{(x,v)} |v|^2 d\xi}$.
Precisely, \cite[Chap.~4]{gigliGeometrySpaceProbability2008} shows the following link: if $\xi \in \Tan_{\mu}$, then its barycenter, defined as the unique $b \in L^2_{\mu}(\Sp;\mathbb{R})$ such that $\int_{(x,v)} \varphi(x) \psi(v) d\xi = \int_x \varphi(x) \psi(b(x)) d\mu$ for $\varphi \in \mathcal{C}_b(\Sp;\mathbb{R})$ and $\psi$ linear, induces an element of $\rTan_{\mu}$. 

Besides the distance $W_{\mu}(\cdot,\cdot)$ adapted to the tangent cone at $\mu$, one can define a metric scalar product
\begin{align*}
	\left<\xi,\zeta\right>_{\mu} 
	\coloneqq \frac{1}{2} \left[\|\xi\|_{\mu}^2 + \|\zeta\|_{\mu}^2 - W_{\mu}^2(\xi,\zeta)\right]
	= \sup_{\eta \in \Gamma_{\mu}(\xi,\zeta)} \int_{(x,v,w)} \left<v,w\right> d\eta(x,v,w).
\end{align*}
The orthogonal $\Tan_{\mu}^{\perp}$, referred to as the set of \emph{solenoidal} measure fields, is defined as the set of $\zeta$ such that $\left<\xi,\zeta\right>_{\mu} = 0$ for any $\xi \in \Tan_{\mu}$.

\paragraph{Submeasure estimate}

The following estimate seems to be folklore, and is included for lack of a precise reference. 

\begin{lem}[Large-mass submeasures of a given measure are close]\label{res:smallestim}
	Let $\mu \in \mathcal{M}_+(\Sp)$ be a nonnegative Borel measure with finite mass, supported on a compact $K$ of diameter $\diam K$. Let $0 \leqslant \varepsilon \leqslant \mu(\Sp)$, and $\alpha,\beta \in \mathcal{M}_+(\Sp)$ satisfy 
	\begin{align*}
		\alpha \leqslant \mu, \qquad
		\beta \leqslant \mu, \qquad 
		\alpha(\Sp) = \beta(\Sp) \geqslant \mu(\Sp) - \varepsilon.
	\end{align*}
	Then for any $p \in [1,\infty)$, there holds $W_p^p(\alpha,\beta) \leqslant \varepsilon (\diam K)^p$. 
\end{lem}

\begin{proof}
	Define $\alpha \wedge \beta$ as the nonnegative measure $(\alpha \wedge \beta)(A) \coloneqq \inf_{(A_n)_n \in \mathfrak{P}(A)} \sum_n \min(\alpha(A_n),\beta(A_n))$, where $\mathfrak{P}(A)$ is the set of measurable countable partitions $(A_n)_n$ of $A$. Then 
	\begin{align*}
		(\alpha \wedge \beta)(\Sp)
		= \inf_{(A_n)_n \in \mathfrak{P}(\Sp)} \sum_n \min(\alpha(A_n),\beta(A_n))
		\geqslant \mu(\Sp) - \sup_{(A_n)_n \in \mathfrak{P}(\Sp)} \sum_n \max((\mu - \alpha)(A_n), (\mu - \beta)(A_n)).
	\end{align*}
	Since $\max(a,b) \leqslant a+b$ whenever $a,b \geqslant 0$, and $(\mu - \alpha)(A_n), (\mu - \beta)(A_n) \geqslant 0$ by assumption, 
	\begin{align*}
		(\alpha \wedge \beta)(\Sp)
		\geqslant \mu(\Sp) - \sup_{(A_n)_n \in \mathfrak{P}(\Sp)} \sum_n \left[(\mu-\alpha)(A_n) + (\mu-\beta)(A_n)\right]
		= \mu(\Sp) - \left[(\mu-\alpha)(\Sp) + (\mu-\beta)(\Sp)\right]
		\geqslant \mu(\Sp) - 2 \varepsilon.
	\end{align*}
	As $\alpha \wedge \beta$ is a submeasure of both $\alpha$ and $\beta$, we can define a transport plan between $\alpha$ and $\beta$ sending $\alpha \wedge \beta$ on itself, and the remaining mass of $\alpha$ on that of $\beta$. The previous estimate shows that $\alpha(\mathbb{R}) - (\alpha \wedge \beta)(\mathbb{R}) \leqslant \varepsilon$, and since both $\alpha$ and $\beta$ are supported on $K$, the mass travels at most $\diam K$. Consequently, $W_p^p(\alpha,\beta) \leqslant \varepsilon \left(\diam K\right)^p$.
\end{proof}

\section{Proof of \Cref{res:charLS1}}\label{sec:LS1}

The strategy for~\Cref{res:charLS1} is to exploit the dual reformulation of the Wasserstein distance. Informally, when the quotient $W(\mu,\mu_h)/h$ approaches its maximal value $1$, the mass of $\mu$ must concentrate on points on which Kantorovich pairs $(\varphi,\varphi^c)$ satisfy a certain 3-point inequality. For the cost $c(x,y) = |x - y|$, this inequality implies that $\varphi$ looks like a downward pointing corner with at least a certain angle: this can happen on at most countably many points, that must be at some controlled distance from each other. For $c(x,y) = |x-y|^p$ with $p > 1$, we have no such interpretation, but computation shows that these points must be sufficiently separated, otherwise the condition $\varphi^c \ominus \varphi \leqslant c$ breaks. The ``separation'' precisely enforces the presence of holes of size comparable to $h$ on which $\mu$ puts very few mass, and intersecting over countably many scales, one gets to the porous sets of \Cref{intro:def:porous}. The converse is based on the construction of a particular Kantorovich potential.

\medskip
We first prove an estimate which is specific to dimension 1. 

\begin{lem}[Pointwise bound on the monotone plan]\label{res:boundR}
	Let $\zeta \in \Psp_2(\T\Sp)_{\mu}$ be such that $|v| \leqslant R$ for some $R > 0$ and $\zeta-$a.e. $(x,v) \in \T\Sp$. Then, the optimal transport plan $\eta \in \Gamma_o(\mu,\exp_{\mu}(\zeta))$ is concentrated on pairs $(x,y)$ such that $|y-x| \leqslant R$.
\end{lem}

\begin{proof}
	Let $\nu \coloneqq \exp_{\mu}(\zeta)$, and denote $F_{\mu},F_{\nu} : \mathbb{R} \to [0,1]$ the respective distribution functions of $\mu,\nu$, i.e. $F_{\mu}(x) = \mu((-\infty,x])$. 
	By the explicit expression of the optimal transport plan in dimension one \cite[Chap.~2]{santambrogioOptimalTransportApplied2015}, $\eta$ is concentrated on pairs $(x,y)$ such that for some $r \in \mathbb{R}$, 
	\begin{align*}
		x = \inf \left\{ \overline{x} \st F_{\mu}(\overline{x}) \geqslant r \right\}, 
		\qquad \text{and} \qquad
		y = \inf \left\{ \overline{y} \st F_{\nu}(\overline{y}) \geqslant r \right\}.
	\end{align*}
	As $F_{\mu},F_{\nu}$ are nondecreasing and upper semi-continuous, the $r$ associated to $(x,y)$ also satisfies $F_{\mu}(x) \geqslant r$ and $F_{\nu}(y) \geqslant r$ (although the inequality can be strict in case of Dirac mass). By assumption, for any $x \in \Sp$, one has 
	\begin{align*}
		F_{\nu}(x + R) = \int_{(y,w)} \ind_{(-\infty,x+R]}(y+w) d\zeta \geqslant \int \ind_{(-\infty,x]}(y) d\mu = F_{\mu}(x) \geqslant r,
	\end{align*}
	hence $y \leqslant x + R$. Symmetrically, one has $\mu = \exp_{\nu}(\gamma)$ for $\gamma \coloneqq (\pi_x + \pi_v, - \pi_v)_{\#} \zeta$, which also satisfies $|v| \leqslant R$ for $\gamma-$a.e. $(x,v)$; hence $F_{\mu}(y + R) \geqslant F_{\nu}(y) \geqslant r$ for any $(x,y)$ in $\supp \eta$ and associated $r$, and we conclude that $x \leqslant y + R$. 
\end{proof}

\begin{rem}\label{rem:limindepp}
	As a consequence, the property \Cref{intro:limsup1} is independent of $p \in [1,\infty)$. Indeed, by Cauchy-Schwarz, the Wasserstein distances are ordered as $W_p(\mu,\nu) \leqslant W_q(\mu,\nu)$ for $p \leqslant q$. On the other hand, by \Cref{res:boundR}, the monotone transport plan $\eta_h$ between $\mu$ and $\mu_h$ satisfies $|y-x| \leqslant h$ almost everywhere. Since this plan is optimal for any $p \in [1,\infty)$, it follows that
	\begin{align*}
		\frac{W_1(\mu,\mu_h)}{h}
		\leqslant \frac{W_p(\mu,\mu_h)}{h}
		= \left(\frac{W_p^p(\mu,\mu_h)}{h^p}\right)^{1/p}
		= \left(\frac{\int |y-x|^p d\eta_h}{h^p}\right)^{1/p}
		\leqslant \left(\frac{h^{p-1} \int |y-x| d\eta_h}{h^p}\right)^{1/p}
		= \left(\frac{W_1(\mu,\mu_{h})}{h}\right)^{1/p}.
	\end{align*}
	Hence, both the set of measures $\mu$ such that \Cref{intro:limsup1} holds, and such that $W_p(\mu,\mu_h) = o(h)$, are independent of $p \in [1,\infty)$.
\end{rem}

The following lemma is the core of the argument of \Cref{res:charLS1}. It states that for any $\varphi$, the points where $\frac{1}{2} \left[\varphi^c(x-h) + \varphi^c(x+h)\right] - \varphi(x)$ is comparable to $h^p$ must be separated by some distance itself comparable to $h$.

\begin{lem}\label{res:coarsePorous}
	Let $p \in [1,\infty)$ and $\kappa_{\gamma} \coloneqq (2 \gamma - 1)^{1/p}$ for any $\gamma \in (1/2,1]$. Given $\varphi : \Sp \to \mathbb{R}$, $\gamma \in (1/2,1]$, and $h > 0$, let 
	\begin{align*}
		S \coloneqq \left\{ x \in \Sp \st \frac{\frac{1}{2} \left[\varphi^c(x-h) + \varphi^c(x+h)\right] - \varphi(x)}{h^p} \geqslant \gamma \right\}.
	\end{align*}
	Then for any $x \in S$, there holds $S \cap \mathscr{B}(x,h(1+\kappa_\gamma)) \subset \overline{\mathscr{B}}(x,h(1-\kappa_{\gamma}))$.
\end{lem}

\begin{proof}
	Assume that $x, y \in S$ with $x \leqslant y$, and let $\delta \coloneqq |y - x|$. By definition of $S$, there holds
	\begin{align*}
		2 \gamma h^p
		\leqslant \frac{\varphi^c(x-h) + \varphi^c(x+h)}{2} - \varphi(x) + \frac{\varphi^c(y-h) + \varphi^c(y+h)}{2} - \varphi(y).
	\end{align*}
	Recall that $\varphi^c(\bdot) \coloneqq \inf_{z \in \Sp} \varphi(z) + |z - \bdot|^p$. Therefore, we can bound the first term by $\psi_{x-h}(z) \coloneqq \varphi(z) + |z - (x-h)|^p$ evaluated at $z = x$, the second term by $\psi_{x+h}(y)$, the third term by $\psi_{y-h}(x)$, and the fourth by $\psi_{y+h}(y)$, to get
	\begin{align*}
		2 \gamma h^p
		\leqslant \frac{\varphi(x) + \varphi(y)}{2} + \frac{h^p + |x+h-y|^p}{2} - \varphi(x) + \frac{\varphi(x) + \varphi(y)}{2} + \frac{|y-h-x|^p + h^p}{2} - \varphi(y).
	\end{align*}
	Rearranging, and using that $2 \gamma - 1 > 0$ since $\gamma > 1/2$, we find
	\begin{align*}
		h^p (2 \gamma - 1) \leqslant |h - \delta|^p
		\qquad \implies \qquad
		h \kappa_{\gamma} = h (2 \gamma - 1)^{1/p} \leqslant |h - \delta|.
	\end{align*}
	The above inequality forces $\delta = |y-x|$ to be either in $[0,h(1-\kappa_{\gamma})]$, or in $[h (1 + \kappa_{\gamma}),\infty)$, so that all points of $S$ that are at distance strictly less than $h (1+\kappa_{\gamma})$ of $x$ must belong to $\overline{\mathscr{B}}(x,h(1-\kappa_{\gamma}))$. 
\end{proof}

\begin{rem}[Taking $\gamma < 1/2$]\label{rem:coarsePorousAnyGamma}
	\Cref{res:coarsePorous} can be proved for arbitrarily small $\gamma > 0$, up to taking more than two points $x,y$: one considers a chain $x_0,x_1,\cdots,x_N$ of points of $S$, sums the defining inequalities, and bounds each $\varphi^c(x_i \pm h)$ by the values at the left and right neighbours of the chain. The corresponding $\kappa$ changes but is still positive.
\end{rem}

\begin{proof}[Proof of \Cref{res:charLS1}]
	We start with the implication \cref{item:charLS1:limit} $\Rightarrow$ \cref{item:charLS1:concen}. Let $(h_n)_{n \in \mathbb{N}} \subset (0,1]$ be a vanishing sequence along which the limit sup is reached. For each $n$, pick $\iota_n \geqslant 0$ going to $0$ with $n$, and then $\varphi_n : \Sp \to \mathbb{R} \cup \{+\infty\}$ be $\iota_n-$optimal in the dual formulation of the Wasserstein distance, i.e. 
	\begin{align}\label{proof:charLS1:cheby}
		1 - \iota_n
		\leqslant \frac{W_p^p(\mu,\exp_{\mu}(h_n \cdot \xi))}{h_n^p} 
		&\leqslant \frac{1}{h_n^p} \int_{y \in \Sp} \varphi_n^c(y) d[\exp_{\mu}(h_n \cdot \xi)] - \frac{1}{h_n^p} \int_{x \in \Sp} \varphi_n(x) d\mu + \iota_n \notag \\
		&= \int_{x \in \Sp} \frac{\frac{1}{2} \left[\varphi_n^c(x-h_n) + \varphi_n^c(x+h_n)\right] - \varphi_n(x)}{h_n^p} d\mu + \iota_n.
	\end{align}
	Let $\varepsilon > 0$, and $(n_m)_{m}$ be a subsequence such that $\iota_{n_m} \leqslant \varepsilon 2^{-2(m+2)}$. For readability, relabel $\iota_{n_m},h_{n_m},\varphi_{n_m}$ as $\iota_m,h_m,\varphi_m$ respectively. Let $\gamma_m \coloneqq 1 - 2^{-(m+2)} \subset (1/2,1)$, and for each $m$, let 
	\begin{align*}
		S_m \coloneqq \left\{ x \in \Omega \st \frac{\frac{1}{2} \left[\varphi_m^c(x-h_m) + \varphi_m^c(x+h_m)\right] - \varphi_m(x)}{h_m^p} \geqslant \gamma_m \right\}.
	\end{align*}
	By definition of the $c$-transform, $\frac{1}{2} \left[\varphi_m^c(x-h_m) + \varphi_n^c(x+h_m)\right] - \varphi_m(x) \leqslant h_m^p$ always holds, so that \Cref{proof:charLS1:cheby} gives
	\begin{align*}
		1 - \iota_m
		\leqslant \int_{S_m} 1 d\mu + \int_{S_m^c} \gamma_m d\mu + \iota_m
		= 1 - \mu(S_m^c) + \gamma_m	\mu(S_m^c) + \iota_m.	
	\end{align*}
	In consequence, $\mu(S_m^c) \leqslant \frac{2 \iota_m}{1-\gamma_m} \leqslant \varepsilon 2^{-(m+1)}$. The measurable set $A \coloneqq \bigcap_{m \in \mathbb{N}} S_m$ satisfies $\mu(A) \geqslant 1 - \sum_{m} \mu(S_m^c) = 1 - \varepsilon$, so we just have to prove that it belongs to $\mathscr{A}$. To this aim, consider the vanishing sequence $s_m \coloneqq h_m (1 + \kappa_m)$ with $\kappa_m = (2\gamma_m - 1)^{1/p}$, and $\tau_m \coloneqq (1-\kappa_m)/(1+\kappa_m)$. By \Cref{res:coarsePorous}, for each $m$, there holds
	\begin{align*}
		A \cap \mathscr{B}(x,s_m)
		\subset S_m \cap \mathscr{B}(x,s_m) 
		\subset \overline{\mathscr{B}}(x,s_m \tau_m)
		\qquad \forall x \in S_m.
	\end{align*}
	Passing to the supremum over $x \in A \subset S_m$, we get
	\begin{align*}
		\sup_{x \in A} \inf \left\{ \tau \in (0,1) \st A \cap \mathscr{B}(x,s_m) \subset \overline{\mathscr{B}}(x,s_m \tau) \right\}
		\leqslant \tau_m
		= \frac{1 - \kappa_m}{1 + \kappa_m}
		= \frac{1 - (2 \gamma_m - 1)^{1/p}}{1 + (2 \gamma_m - 1)^{1/p}}.
	\end{align*}
	Since $\gamma_m \to_m 1$, we may pass to the limit in $m \to \infty$ and conclude that $A \in \mathscr{A}$.
	
	Assume now that \Cref{item:charLS1:concen} holds. We prove \cref{item:charLS1:limit} for $p=1$, and deduce it for any $p \in [1,\infty)$ by \Cref{rem:limindepp}. Let $\varepsilon > 0$ and $A \in \mathscr{A}$ such that $\mu(A) \geqslant 1 - \varepsilon$. By definition, there exist vanishing sequences $(s_n)_n,(\tau_n)_n \subset (0,1)$ such that $A \cap \mathscr{B}(x,s_n) \subset \overline{\mathscr{B}}(x,s_n \tau_n)$ for any $n \in \mathbb{N}$ and $x \in A$. Let $\psi : x \mapsto d(x,A) = \inf_{z \in A} d(x,z)$ be the (1-Lipschitz) distance function to $A$. Then, by the dual formula~\Cref{dual:p=1},
	\begin{align*}
		\frac{W_1(\mu,\exp_{\mu}(h \cdot \xi))}{h} 
		\geqslant \frac{1}{h} \left[\int_{y \in \Sp} \psi(y) d\exp_{\mu}(h \cdot \xi) - \int_{x \in \Sp} \psi(x) d\mu\right]
		= \int_{x \in \Sp} \frac{\frac{1}{2} \left[\psi(x+h) + \psi(x-h)\right] - \psi(x)}{h} d\mu.
	\end{align*}
	Taking $h_n \coloneqq \frac{s_n + \tau_n s_n}{2}$, we know that for any $x \in A$, the points $x \pm h_n$ are at distance at least $\frac{s_n (1 - \tau_n)}{2}$ of any point of $A$. Hence $\psi(x \pm h_n) \geqslant \frac{s_n (1 - \tau_n)}{2}$ on $A$. On the other hand, one has $\left|\frac{1}{2} \left[\psi(x+h) + \psi(x-h)\right] - \psi(x)\right|/h \leqslant 1$ since $\psi$ is 1-Lipschitz. Hence
	\begin{align*}
		\int_{x \in A \cup A^c} \frac{\frac{1}{2} \left[\psi(x+h_n) + \psi(x-h_n)\right] - \psi(x)}{h_n} d\mu 
		\geqslant \int_{x \in A} \frac{s_n (1 - \tau_n)}{2 h_n} d\mu - \mu(A^c)
		\geqslant \mu(A) \frac{1 - \tau_n}{1 + \tau_n} - \varepsilon.
	\end{align*}
	Taking the limit in $n \to \infty$, we get that 
	\begin{align*}
		\limsup_{n \to \infty} \frac{W_1(\mu, \exp_{\mu}(h \cdot \xi))}{h}
		\geqslant \mu(A) - \varepsilon 
		\geqslant 1 - 2 \varepsilon
		= \|\xi\|_{\mu} - 2 \varepsilon.
	\end{align*}
	Letting $\varepsilon \searrow 0$, we conclude that \Cref{item:charLS1:limit} holds. 
\end{proof}

\section{Link with the geometric tangent cone to $\Psp_2(\Sp)$}\label{sec:linkTan}

In this section, we fix $p = 2$ and denote $W(\cdot,\cdot)$ the corresponding Wasserstein distance. Our aim is to characterize the measures $\mu \in \Psp_2(\Sp)$ for which $\xi \in \Tan_{\mu}$ if and only if 
\begin{align}\label{limitsupcondition}
	\lim_{h \searrow 0} \frac{W(\mu,\exp_{\mu}(h \cdot \xi))}{h} = \|\xi\|_{\mu}.
\end{align}
The main difficulty is to show that whenever some $\xi \notin \Tan_{\mu}$ does satisfy \Cref{limitsupcondition}, then some nonatomic submeasure $\sigma$ of $\mu$ also does with the symmetric unit measure field $\frac{1}{2} \left[(id,-1)_{\#} \sigma + (id,1)_{\#} \sigma\right]$.

\subsection{Reduction to centred measure fields}

We start by proving that whenever $\zeta \in \Psp_2(\T\Sp)_{\mu}$ satisfies \Cref{limitsupcondition}, then the centred field $\zeta^0 \coloneqq (\pi_x, \pi_v - \bary{\xi}{}(\pi_x))_{\#} \zeta$, obtained by removing the barycenter of $\zeta$, also does. The intuition behind is that map-induced and centred measure fields are orthogonal for the metric scalar product, i.e. for any $b \in L^2_{\mu}$, there holds
\begin{align*}
	\left<\zeta^0, (id,b)_{\#} \mu\right>_{\mu}
	= \int_{x \in \Sp} \int_{v \in \T_x\Sp} \left<v,b(x)\right> d\zeta_x^0(v) d\mu(x)
	= 0.
\end{align*}
The same computation shows that $\|\zeta\|_{\mu}^2 = \|\zeta^0\|_{\mu}^2 + \|b\|_{L^2_{\mu}}^2$. This ``infinitesimal'' orthogonality appears in a weaker form before the limit, as shown by the following lemma.

\begin{lem}\label{res:removebary}
	Let $\zeta = (\pi_x,\pi_v + b(\pi_x))_{\#} \zeta^0$ for some $\zeta^0$ with null barycenter. Then 
	\begin{align}\label{removebary:statement}
		\limsup_{h \searrow 0} \frac{W^2(\mu,\exp_{\mu}(h \cdot \zeta))}{h^2}
		\leqslant \limsup_{h \searrow 0} \frac{W^2(\mu,\exp_{\mu}(h \cdot \zeta^0))}{h^2} + \|b\|_{\mu}^2.
	\end{align}
	Consequently, if the left hand-side of \Cref{removebary:statement} is equal to $\|\zeta\|_{\mu}^2 = \|\zeta^0\|_{\mu}^2 + \|b\|_{L^2_{\mu}}^2$, then the centred measure field $\zeta^0$ satisfies the limit sup condition~\Cref{limitsupcondition}.
\end{lem}

\begin{proof}
	
	For any $\xi,\zeta \in \Psp_2(\T\Sp)_{\mu}$, there holds 
	\begin{align*}
		\left|\limsup_{h \searrow 0} \frac{W(\mu,\exp_{\mu}(h \cdot \xi))}{h} - \limsup_{h \searrow 0} \frac{W(\mu,\exp_{\mu}(h \cdot \zeta))}{h}\right|
		\leqslant W_{\mu}(\xi,\zeta)
		= \sqrt{W_{\mu}^2(\xi^0,\zeta^0) + \|b_{\xi} - b_{\zeta}\|_{L^2_{\mu}}^2}.
	\end{align*}
	Hence both sides of \Cref{removebary:statement} are continuous with respect to $\zeta^0$ and $b$, respectively in the $W_{\mu}-$ and $L^2_{\mu}-$topologies. Consequently, we might prove \Cref{removebary:statement} in the case where $|v| \leqslant R$ for $\zeta^0-$a.e. $(x,v)$ and some $R > 0$, and $b = \sum_{i=1}^N b_i \ind_{C_i}$ for a partition $(C_i)_{i=1}^N$ of $\Sp$ into finitely many intervals such that $\mu(\partial C_i) = 0$ for each $i$. 
	
	Denote $\mu_i \coloneqq \mu \resmes C_i$, with $\zeta_i \coloneqq \zeta \resmes \{(x,v) \, |\, x \in C_i\}$, and $\zeta_i^0$ defined in the same way. Then $\mu = \sum_i \mu_i$, and by linearity of the pushforward, $\exp_{\mu}(h \cdot \zeta) = \sum_i \exp_{\mu_i}(h \cdot \zeta_i)$. 	
	By the convexity of the squared Wasserstein distance in the Banach space of measures, 
	\begin{align*}
		W^2(\mu,\exp_{\mu}(h \cdot \zeta))
		\leqslant \sum_{i=1}^N W^2(\mu_i, \exp_{\mu_i}(h \cdot \zeta_i)).
	\end{align*}
	By definition of $C_i$, $\zeta_i = (\pi_x,\pi_v + b_i)_{\#} \zeta_i^0$, with $b_i \in \mathbb{R}$ a constant. Let $\eta = \eta(dx,dy) \in \Gamma_o(\mu_i,\exp_{\mu_i}(h \cdot \zeta_i^0))$. Then $(\pi_x, \pi_y + h b_i)_{\#} \eta$ is a transport plan between $\mu_i$ and $\exp_{\mu}(h \cdot \zeta_i)$, so that
	\begin{align*}
		W^2(\mu_i, \exp_{\mu_i}(h \cdot \zeta_i)) 
		\leqslant \int_{(x,y) \in \Sp^2} |x - (y + h b_i)|^2 d\eta 
		= W^2(\mu_i, \exp_{\mu_i}(h \cdot \zeta_i^0)) - 2 \int_{(x,y) \in \Sp^2} \left<x - y, h b_i\right> d\eta + \mu(C_i) |h b_i|^2.
	\end{align*}
	Since $\int y d\eta = \int (x + h v) d\zeta_i^0 = \int x d \mu + 0$, the middle term vanishes. Summing over $i$, we obtain
	\begin{align}\label{removebary:firstestimate}
		W^2(\mu,\exp_{\mu}(h \cdot \zeta))
		\leqslant \sum_{i=1}^N W^2(\mu_i, \exp_{\mu_i}(h \cdot \zeta_i^0)) + h^2 \|b\|_{\mu}^2.
	\end{align}
	Stays to show that $\sum_{i=1}^N W^2(\mu_i, \exp_{\mu_i}(h \cdot \zeta_i^0)) \leqslant W^2(\mu,\exp_{\mu}(h \cdot \zeta^0)) + h^2 O(h)$. To this aim, let $\eta \in \Gamma_o(\mu,\exp_{\mu}(h \cdot \zeta^0))$, and decompose it as $\sum_i \eta_{i}$ for $\eta_{i} \coloneqq \eta_h \resmes C_i \times \Sp$. Let also $\nu_i \coloneqq \pi_{y \#} \eta_{i}$. Since $\sum_i \nu_i = \exp_{\mu}(h \cdot \zeta^0) = \sum_i \exp_{\mu_i}(h \cdot \zeta_i^0)$, there holds
	\begin{align}\label{removebary:sumFi}
		\sum_{i=1}^N F_{\nu_i} = F_{\exp_{\mu}(h \cdot \zeta^0)} = \sum_{i=1}^N F_{\exp_{\mu_i}(h \cdot \zeta_i^0)}.
	\end{align}
	Here we recall that $F_{\sigma}(x) \coloneqq \sigma((-\infty,x])$ for any measure $\sigma$. 
	However, by \Cref{res:boundR}, the plan $\eta$ is supported on pairs $(x,y)$ with $|y-x| \leqslant h R$. Consequently, both $\nu_i$ and $\exp_{\mu_i}(h \cdot \zeta_i^0)$ are supported on points that are at distance at most $h R$ from $C_i$. This implies that $F_{\nu_i}(x) = F_{\exp_{\mu_i}(h \cdot \zeta_i^0)}(x) = 0$ for $x \leqslant \min C_i - h R$, and $\mu(C_i)$ for $x \geqslant \max C_i + h R$. Hence, evaluating \Cref{removebary:sumFi} at $x \in C_i$ that is at distance at least $h R$ from the boundary $\partial C_i$, there holds $F_{\nu_i}(x) = F_{\exp_{\mu_i}(h \cdot \zeta_i^0)}(x)$. Using the expression of the monotone optimal plan, we deduce that the optimal transport plan $\alpha_i$ between $\nu_i$ and $\exp_{\mu_i}(h \cdot \zeta_i^0)$ is concentrated on pairs $(x,y)$ such that if $x \in [\min C_i + h R, \max C_i - h R]$, then $y = x$. In addition, $|y-x| \leqslant 2 h R$, since mass only moves within balls of radius $h R$ around $\min C_i$ and $\max C_i$. It follows that
	\begin{align*}
		W^2(\nu_i, \exp_{\mu_i}(h \cdot \zeta_i^0))
		&= \int_{(x,y)} |y-x|^2 d\alpha_i
		= \int_{(x,y),|x - \min C_i| \leqslant h R \text{ or } |x - \max C_i| \leqslant h R} |y-x|^2 d\alpha_i \\
		&\leqslant (2 h R)^2 \mu\left(\overline{\mathscr{B}}(\min C_i, h R) \cup \overline{\mathscr{B}}(\max C_i, h R)\right).
	\end{align*}
	Since $\mu(\partial C_i) = 0$, the above quantity is of order $h^2 O(h)$. Now, by the second triangular inequality, 
	\begin{align*}
		W^2(\mu,\exp_{\mu}(h \cdot \zeta^0))
		&= \sum_{i=1}^N W^2(\mu_i, \nu_i) 
		\geqslant \sum_{i=1}^N W^2(\mu_i, \exp_{\mu_i}(h \cdot \zeta_i^0)) - W(\nu_i, \exp_{\mu_i}(h \cdot \zeta_i^0))\left(W(\nu_i, \exp_{\mu_i}(h \cdot \zeta_i^0)) + 2 W(\mu_i, \nu_i)\right) \\
		&\geqslant \sum_{i=1}^N W^2(\mu_i, \exp_{\mu_i}(h \cdot \zeta_i^0)) - W(\nu_i, \exp_{\mu_i}(h \cdot \zeta_i^0)) 3 h R \mu(C_i).
	\end{align*}
	Combining the two last lines provides the desired estimate in \Cref{removebary:firstestimate}, and we conclude.
\end{proof}

\subsection{Reduction to the unit symmetric measure field}

We now prove that whenever \emph{some} centred measure field $\xi \notin \Tan_{\mu}$ satisfies \Cref{limitsupcondition}, then $\mu$ admits a nonatomic submeasure on which the unit symmetric measure field also satisfies \Cref{limitsupcondition}. The argument uses various operations to modify a measure field and still satisfy \Cref{limitsupcondition}; for readability, we introduce the notation
\begin{align*}
	\Xi[\mu] \coloneqq \left\{ \zeta \in \Psp_2(\T\Sp)_{\mu} \st \bary{\zeta}{} = 0 \ \ \text{ and } \ \ \limsup_{h \searrow 0} \frac{W(\mu,\exp_{\mu}(h \cdot \zeta))}{h} = \|\zeta\|_{\mu} \right\}.
\end{align*}
If $\mu$ is a nonnegative Borel measure with arbitrary finite mass $m > 0$, we still denote $\Xi[\mu] \coloneqq m \Xi[\mu/m]$.

\begin{lem}[Technical manipulations]\label{res:technicalmanip}
	Let $\mu \in \Psp_2(\Sp)$. Then:
	\begin{itemize}
	\item (closure) the set $\Xi[\mu]$ is closed with respect to $W_{\mu}$.
	\item (nonnegative cone) $\zeta \in \Xi[\mu]$ implies that $(\pi_x, \lambda \pi_v)_{\#} \zeta \in \Xi[\mu]$ for any $\lambda \geqslant 0$.
	\item (stability by restriction) For any measurable $A$ with $\mu(A) > 0$, $\zeta \in \Xi[\mu]$ implies that $\zeta \resmes \{(x,v) \, |\, x \in A\} \in \Xi[\mu \resmes A]$.
	\item (stability by truncation) If $f \in L^2_{\mu}(\Sp;\mathbb{R}^+)$ is such that $\frac{1}{2}\left[(id,-f)_{\#} \mu + (id,f)_{\#} \mu\right] \in \Xi[\mu]$, then for any $R > 0$, there holds $\frac{1}{2}\left[(id,-\min(f,R))_{\#} \mu + (id,\min(f,R))_{\#} \mu\right] \in \Xi[\mu]$.
	\end{itemize}
\end{lem}

\begin{proof}
	The set of measure fields with 0 barycenter is stable by all the operations involved. 
	The first and second point are direct, since the defining inequality is continuous with respect to $W_{\mu}$ and positively homogeneous. To show the third point, denote $\zeta_A \coloneqq \zeta \resmes \{(x,v) \, |\, x \in A\}$. Then, using the convexity of the squared Wasserstein distance in the Banach space of measures \cite[Theorem~4.8]{villaniOptimalTransport2009},
	\begin{align*}
		\int_{(x,v), x \in A} |v|^2 d\zeta + \int_{(x,v), x \in A^c} |v|^2 d\zeta
		&= \|\zeta\|_{\mu}^2 
		\leqslant \limsup_{h \searrow 0} \frac{W^2(\mu, (\pi_x + h \pi_v)_{\#} \zeta)}{h^2} \\
		&\leqslant \limsup_{h \searrow 0} \frac{W^2(\mu \resmes A, (\pi_x + h \pi_v)_{\#} \zeta_A) + W^2(\mu \resmes A^c, (\pi_x + h \pi_v)_{\#} \zeta_{A^c})}{h^2} \\
		&\leqslant \limsup_{h \searrow 0} \frac{W^2(\mu \resmes A, (\pi_x + h \pi_v)_{\#} \zeta_A)}{h^2} + \int_{(x,v), x \in A^c} |v|^2 d\zeta.
	\end{align*}
	As converse inequality always holds, we deduce that $\zeta_A \in \Xi[\mu \resmes A]$. 
	
	The last point is reached by a limit procedure. Let $\varepsilon > 0$. Define a sequence $(f_n)_{n \in \mathbb{N}} \subset L^2_{\mu}$ by $f_0 \coloneqq f$, and the following construction. For each $n$, let $A \coloneqq \{f_n \geqslant R + \varepsilon\}$, and denote $\gamma_n \coloneqq \frac{1}{2} \left[(id,-f_n)_{\#} \mu + (id,f_n)_{\#} \mu\right]$. We construct a transport plan $\alpha \in \Gamma_{\mu}(\gamma_n, \gamma_n)$ that sends $v$ to $-v$ on $A$, and $v$ to $v$ on $A^c$, i.e.
	\begin{align*}
		\alpha \coloneqq \frac{1}{2} \left(id, f_n, f_n \ind_{A^c} - f_n \ind_A \right)_{\#} \mu + \frac{1}{2} \left(id, -f_n, f_n \ind_{A} - f_n \ind_{A^c} \right)_{\#} \mu.
	\end{align*}
	Let $t \coloneqq R/(R+\varepsilon) \in (0,1)$, and $\gamma_{n+1} \coloneqq (\pi_x, (1-t) \pi_v + t \pi_w)_{\#} \alpha$. 
	By the displacement semiconcavity of the squared Wasserstein distance \cite[Theorem 7.3.2]{ambrosioGradientFlows2005}, there holds
	\begin{align*}
		W^2\left(\mu, \exp_{\mu}(h \cdot \gamma_{n+1})\right)
		\geqslant (1-t) W^2(\mu,\exp_{\mu}(h \cdot \gamma_n)) + t W^2(\mu,\exp_{\mu}(h \cdot \gamma_n)) - t (1-t) h^2 \int_{(x,v,w)} |v - w|^2 d\alpha.
	\end{align*}
	(Precisely, we took $\mu^1 = \mu^2 = \exp_{\mu}(h \cdot \gamma_n)$ and $\mu^3 = \mu$, along with the plan $\boldsymbol{\mu}^{1\,2} = (\pi_x + h \pi_v, \pi_x + h \pi_w)_{\#} \alpha$. Then, still using the notation of \cite{ambrosioGradientFlows2005}, $W_{\boldsymbol{\mu}^{1\,2}}^2(\mu^1,\mu^2) = \int |(x + h v) - (x + h w)|^2 d\alpha = h^2 \int |v-w|^2 d\alpha$.) Dividing by $h^2 > 0$ and taking the limit sup in $h \searrow 0$, we get that
	\begin{align*}
		\limsup_{h \searrow 0} \frac{W^2\left(\mu, \exp_{\mu}(h \cdot \gamma_{n+1})\right)}{h^2}
		\geqslant \|\gamma_n\|_{\mu}^2 - t (1-t) \int_{(x,v,w)} |v - w|^2 d\alpha
		= \int_{(x,v,w)} |(1-t) v + t w|^2 d\alpha
		= \|\gamma_{n+1}\|_{\mu}^2.
	\end{align*}
	Here we used that $(1-t) |v|^2 + t |w|^2 - t (1-t) |v - w|^2 = |(1-t) v + t w|^2$ for any vectors $v,w$. Therefore, the measure field $\gamma_{n+1}$ belongs to $\Xi[\mu]$. As $n$ goes to $\infty$, the map $f_n$ converges in $\mu-$measure towards some function that coincides with $f$ whenever $f \leqslant R$, and lies in $[R,R+\varepsilon]$ otherwise. Taking $\varepsilon > 0$ small enough, then $n$ large enough, we can approximate $\min(f,R)$ in $L^2_{\mu}$ with arbitrarily small error. Since convergence of maps in $L^2_{\mu}$ implies convergence of their associated symmetric measure field in $W_{\mu}$, and since $\Xi[\mu]$ is $W_{\mu}-$closed by the first point, we conclude. 
\end{proof}

Lastly, we turn to the main step of the reduction. The strategy is to first restrict to a nonatomic submeasure of $\mu$ on which $\zeta$ splits mass; then to show that taking averages of the ``positive'' and ``negative'' parts does not break the limit sup condition~\Cref{limitsupcondition}; then to normalize the obtained measure field.

\begin{lem}\label{res:reduc}
	Assume that $\zeta \in \Xi[\mu]$ for some $\zeta \notin \Tan_{\mu}$. Then there exists a nonatomic measure $\sigma \leqslant \mu$ of positive mass such that $\xi \in \Xi[\sigma]$ for the symmetric unit measure field $\xi \coloneqq \frac{1}{2} \left[(id,-1)_{\#} \sigma + (id,1)_{\#} \sigma\right]$.
\end{lem}

\begin{proof}
	\newcounter{step}
	We modify $\zeta$ in successive steps to obtain a symmetric unit measure field, and show that each steps preserves the corresponding set $\Xi$.
	
	\paragraph{Step~\thestep: restriction to a suitable set} \stepcounter{step}
	
	Decompose $\mu = \mu^a + \mu^d$, with $\mu^a$ purely atomic and $\mu^d$ atomless.
	By \cite[Theorem~3.7]{sedjroCharacterizationTangentSpace2014} (see also \cite[Proposition~2.9]{aussedatStructureGeometricTangent2025}), the tangent cone to $\mu$ is made of measure fields $\gamma$ of the form $\gamma^a + (id,f)_{\#} \mu^{d}$, where $\pi_{x \#} \gamma^a = \mu^a$ is a measure field on $\T\Sp$ with first marginal (hence same mass as) the atomic part $\mu^a$, and $f \in L^2_{\mu^d}(\Sp;\mathbb{R})$. Hence $\zeta \notin \Tan_{\mu}$ if and only if $\zeta = \zeta^a + \zeta^d$, with $\pi_{x \#} \zeta^a = \mu^a$, $\pi_{x \#} \zeta^d = \mu^d$ and $\zeta^d$ not induced by a map. In consequence, there exists a measurable set $A \subset \mathbb{R}$ such that 
	\begin{align*}
		\mu^a(A) = 0, \qquad
		\mu^d(A) > 0, \qquad
		\text{and for } \mu-\text{a.e. } x \in A, \quad
		\int_{v} |v|^2 d\zeta_x > 0.
	\end{align*}
	By \Cref{res:technicalmanip}, the restricted measure field $\zeta_A \coloneqq \zeta \resmes \{(x,v) \,|\, x \in A\}$ belongs to $\Xi[\mu \resmes A]$. For the sequel, denote $\nu \coloneqq \mu \resmes A$.
	
	\paragraph{Step~\thestep: reduction to a symmetric measure field} \stepcounter{step}
	
	Now, split $\zeta_A$ in a ``left'' and ``right'' part as follows: consider $\alpha = \alpha(dx,dv,dw) \in \Gamma_{\nu,o}(\zeta_A,\xi)$, with again $\xi \coloneqq \frac{1}{2} \left[(id,-1)_{\#} \nu + (id,1)_{\#} \nu\right]$ the symmetric unit field. Then $\alpha = \frac{1}{2} \left[\alpha^- + \alpha^+\right]$, where $(\pi_x,\pi_w)_{\#} \alpha^{\pm} = (id,\pm 1)_{\#} \nu$. Using the other marginal condition, we can write $\zeta_A = (\pi_x, \pi_v)_{\#} \alpha = \frac{1}{2} (\pi_x,\pi_v)_{\#} \alpha^- + \frac{1}{2} (\pi_x,\pi_v)_{\#} \alpha^+ \eqqcolon \frac{1}{2} \left[\zeta_A^- + \zeta_A^+\right]$. 
	
	Denote $f^{\pm} \in L^2_{\nu}$ the respective barycenters of $\zeta_A^{\pm}$. Then $\frac{1}{2} \left[f^- + f^+\right]$ is the barycenter of $\zeta_A$, i.e. $0$, hence $f^- = - f^+$. We show that $f^+(x) > 0$ for $\nu-$a.e. $x \in \Sp$. By construction, $\alpha = (\alpha_x)_x \otimes \nu$, with $\alpha_x$ identified with an optimal plan between $(\zeta_A)_x$ and $\xi_x = \frac{1}{2} \left[\delta_{-1} + \delta_1\right]$. Hence
	\begin{align*}
		\frac{1}{2} \left[\int_v |v|^2 d\zeta_A + \int_w |w|^2 d\xi_x - W^2((\zeta_A)_x, \xi_x)\right]
		&= \int_{(v,w)} \frac{1}{2} \left[|v|^2 + |w|^2 - |v - w|^2\right] d\alpha_x
		= \int_{(v,w)} \left<v,w\right> d\alpha_x \\
		&= \frac{1}{2} \int_{(v,w)} \left<v,-1\right> d\alpha_x^{-} + \frac{1}{2} \int_{(v,w)} \left<v,1\right> d\alpha_x^+
		= f^+(x).
	\end{align*}
	By \cite[Lemma 2.9]{aussedatLocalityCentredTangent2025}, the leftmost term is always nonnegative, and 0 if and only if $(\zeta_A)_x = \delta_0$, which may happen only on a $\nu-$negligible set by assumption. Hence $f^+(x) > 0$ for $\nu-$a.e. $x \in \Sp$. 
	
	We now replace $\zeta_A$ by $\gamma \coloneqq \frac{1}{2} \left[(id,f^-)_{\#} \nu + (id,f^+)_{\#} \nu\right]$ as follows: on the one hand,
	\begin{align}\label{reduc:twosided:norm}
		\int_{x \in \Sp} |v|^2 d\zeta_A 
		&= \frac{1}{2} \left[\int_{x \in \Sp} |f^-(x)|^2 d\nu + \int_{x \in \Sp} |f^+(x)|^2 d\nu\right] + \frac{1}{2} \left[\int_{(x,v)} |v - f^-(x)|^2 d\zeta_A^- + \int_{(x,v)} |v - f^+(x)|^2 d\zeta_A^+\right] \notag \\
		&= \|\gamma\|_{\nu}^2 + \frac{1}{2} \sum_{s \in \{-,+\}} \int_{(x,v)} |v - f^s(x)|^2 d\zeta_A^s.
	\end{align}
	On the other hand, using the dual reformulation, there holds
	\begin{align}\label{reduc:twosided:dual}
		W^2(\nu,\exp_{\nu}(h \cdot \zeta_A))
		= \sup_{\varphi \in L^1_{\mu}} \int_{(x,v) \in \T\Sp} \varphi^c(x + h v) d \zeta_A - \int_{x \in \Sp} \varphi(x) d\mu.
	\end{align}
	For each fixed $\varphi$ and $z \in \Sp$, the function $y \mapsto \varphi^c(y) - |y - z|^2 = \inf_{w \in \Sp} \varphi(w) + |w - y|^2 - |y - z|^2$ is concave. Hence, taking $z = x + h f^{\pm}(x)$ inside the integral, we can write that
	\begin{align*}
		&\int_{(x,v) \in \T\Sp} \varphi^c(x + h v) d \zeta_A \\
		&= \frac{1}{2} \sum_{s \in \{-,+\}} \int_{(x,v) \in \T\Sp} \left[\varphi^c(x + h v) - |(x + h v) - (x + h f^{s}(x))|^2\right] d\zeta_A^s + \frac{1}{2} \sum_{s \in \{-,+\}} \int_{(x,v)} |(x + h v) - (x + h f^{s}(x))|^2 d\zeta_A^{s} \\
		&\leqslant \frac{1}{2} \sum_{s \in \{-,+\}} \int_{x \in \Sp} \left[\varphi^c(x + h f^s(x)) - 0\right] d\nu + \frac{h^2}{2} \sum_{s \in \{-,+\}} \int_{(x,v)} |v - f^{s}(x)|^2 d\zeta_A^{s}.
	\end{align*}
	In the first term of the last line, we recognize $\int_{y \in \Sp} \varphi^c(y) d\exp_{\nu}(h \cdot \gamma)$. Plugging this inequality in~\Cref{reduc:twosided:dual}, we obtain
	\begin{align*}
		W^2(\nu,\exp_{\nu}(h \cdot \zeta_A))
		&\leqslant \sup_{\varphi \in L^1_{\mu}} \int_{y \in \Sp} \varphi^c(y) d\exp_{\nu}(h \cdot \gamma) - \int_{x \in \Sp} \varphi(x) d\mu + \frac{h^2}{2} \sum_{s \in \{-,+\}} \int_{(x,v)} |v - f^{s}(x)|^2 d\zeta_A^{s} \\
		&= W^2(\nu,\exp_{\nu}(h \cdot \gamma)) + \frac{h^2}{2} \sum_{s \in \{-,+\}} \int_{(x,v)} |v - f^{s}(x)|^2 d\zeta_A^{s}.
	\end{align*}
	Dividing by $h > 0$ and taking the limit sup in $h \searrow 0$, we get from \Cref{reduc:twosided:norm} that $\gamma \in \Xi[\nu]$.
	
	\paragraph{Step~\thestep: reduction to a unit symmetric measure field} \stepcounter{step}
	
	At this point, we know that there exists $f \coloneqq f^+ \in L^2_{\nu}(\Sp;\mathbb{R}^+)$ such that $f(x) > 0$ for $\nu-$a.e. point, and the measure field $\gamma \coloneqq \frac{1}{2} \left[(id,-f)_{\#} \nu + (id,f)_{\#} \nu\right]$ belongs to $\Xi[\nu]$. By \Cref{res:technicalmanip}, we might multiply $f$ by some large $\lambda > 0$ so that $\nu \left\{ \lambda f \geqslant 1 \right\} \geqslant \nu(\Sp) / 2$, and further restrict to the measurable set $B \coloneqq \{ \lambda f \geqslant 1\} \subset \mathbb{R}$, while still maintaining $\vartheta \in \Xi[\sigma]$ for $\sigma \coloneqq \nu \resmes B$ and $\vartheta \coloneqq \frac{1}{2} \left[(id,-\lambda f)_{\#} \sigma + (id,\lambda f)_{\#} \sigma\right]$. Since $\min(\lambda f,1) = 1$ for $\sigma-$almost every point, we can apply the last point of \Cref{res:technicalmanip} to obtain that the unit symmetric measure field does belong to $\Xi[\sigma]$. As $\sigma = \mu \resmes (A \cap B) \leqslant \mu$, this completes the proof. 
\end{proof}

\begin{rem}
	If we were interested into measure fields for which the \emph{limit} exists, instead of the limit sup, then a shorter proof could be done: one shows that the set $\Xi$ of centred $\xi \in \Psp_2(\T\Sp)_{\mu}$ such that $\lim_{h \searrow 0} W(\mu,\exp_{\mu}(h \cdot \xi))/h = \|\xi\|_{\mu}$ is a $W_{\mu}-$closed convex cone of centred measure fields, in the sense given to it in \cite{aussedatLocalityCentredTangent2025}. By Proposition~2.1 therein, there exists a measurable application $D : x \mapsto D(x)$, where $D(x)$ is either $\{0\}$ or $\mathbb{R}$, and such that $\Xi$ is exactly given by the set of measure fields concentrated on the graph of $D$. The assumption of \Cref{res:reduc} implies that $D(x) = \mathbb{R}$ on some set $A$ of positive $\mu^d-$mass, and one can directly take $\nu \coloneqq \mu^d \resmes A$. However, the dependence to a particular subsequence prevents this argument, and the above proof manually goes through the argument of \cite{aussedatLocalityCentredTangent2025}.
\end{rem}

\subsection{Proof of \Cref{res:answerQ1}}

One of the implications of \Cref{res:answerQ1} follows from the successive reductions of the previous sections. For the other implication, one needs to show that the mass that $\mu$ puts around a set $A \in \mathscr{A}$ does not play a significant role for small $h > 0$.

\begin{proof}
	Assume that $\zeta \notin \Tan_{\mu}$ satisfies $\limsup_{h \searrow 0} W_2(\mu,\exp_{\mu}(h \cdot \zeta))/h = \|\zeta\|_{\mu}$. By \Cref{res:removebary}, the centred measure field $\zeta^0 \coloneqq (\pi_x, \pi_v - \bary{\zeta}{}(\pi_x))_{\#} \zeta$ also does. \vphantom{$\limsup_{h \searrow 0}$}By \Cref{res:reduc}, there exists a submeasure $\sigma \leqslant \mu^d$ such that \Cref{limitsupcondition} holds, i.e. $W_2(\sigma,\sigma_h)/h \to 1$ when $h \searrow 0$. By \Cref{res:charLS1}, $\sigma$ gives mass to an element of $\mathscr{A}$, and so does $\mu^d$.
	
	Conversely, assume that $\mu^d$ puts mass on some $B \in \mathscr{A}$. Since any subset $A \subset B$ also belongs to $\mathscr{A}$, we can choose a compact $A \in \mathscr{A}$ such that $\mu^d(A) > 0$. Let 
	\begin{align*}
		\zeta \coloneqq \frac{1}{2} \left[(id,-1)_{\#} \mu \resmes A + (id,1)_{\#} \mu \resmes A\right] + (id,0)_{\#} \mu \resmes A^c.
	\end{align*}
	We claim that $\limsup_{h \searrow 0} W(\mu,\exp_{\mu}(h \cdot \zeta)) / h = \|\zeta\|_{\mu} = \sqrt{\mu(A)}$: the strategy is to show that when passing to the limit, the neighbourhood of the compact $A$ receives very few mass, and the quantity $W(\mu,\exp_{\mu}(h \cdot \zeta))$ is very close to the sum of the contribution of $\mu \resmes A$ and $\mu \resmes A^c$. 
	
	Denote by $\eta$ the optimal transport plan between $\mu$ and $\exp_{\mu}(h \cdot \zeta)$, which is concentrated on pairs $(x,y)$ such that $|y-x| \leqslant h$ by \Cref{res:boundR}. Let $\mu_A \coloneqq \mu \resmes A$ and $\eta_A \coloneqq \eta \resmes (A \times \Sp)$. Then, by restriction of optimality \cite[Theorem~4.6]{villaniOptimalTransport2009} and the second triangular inequality,
	\begin{align}\label{fromAtolim:secondtri}
		W^2(\mu,\exp_{\mu}(h \cdot \zeta))
		&= \int_{(x,y)} |y-x|^2 d\eta
		\geqslant \int_{(x,y), x \in A} |y-x|^2 d\eta
		= W^2(\mu_A, \pi_{y \#} \eta_A) \notag \\
		&\geqslant W^2(\mu_A, (\mu_A)_h) - W(\pi_{y \#} \eta_A, (\mu_A)_h) \left(2 W(\mu_A, \pi_{y \#} \eta_A) + W(\pi_{y \#} \eta_A, (\mu_A)_h)\right) \notag \\
		&\geqslant W^2(\mu_A, (\mu_A)_h) - W(\pi_{y \#} \eta_A, (\mu_A)_h) \times \left(2 h + (h + h)\right).
	\end{align}
	Let $(s_n)_{n \in \mathbb{N}},(\tau_n)_{n \in \mathbb{N}} \subset (0,1)$ be the vanishing sequences associated to $A$, i.e. such that for any $x \in A$ and $n \in \mathbb{N}$, the intersection $A \cap \mathscr{B}(x,s_n)$ is contained in $\overline{\mathscr{B}}(x, \tau_n s_n)$. Since $A$ is compact, we can pick successively a finite set of points $\{b_1,\cdots,b_K\} \subset A$ such that $A \subset \bigcup_{k=1}^K \mathscr{B}(b_k,s_n)$, and for any $k \neq \ell$, there holds $b_k \notin \mathscr{B}(b_{\ell},s_n)$. Then, the closed balls $B_k \coloneqq \overline{\mathscr{B}}(b_k,\tau_n s_n)$ cover $A$, and are separated from each other by at least $s_n - \tau_n s_n$. 
	
	Consider $h_n \coloneqq \frac{s_n - \tau_n s_n - \varepsilon_n}{2}$, where $\varepsilon_n > 0$ and $\varepsilon_n/h_n \to_n 0$. For this choice, the inflated balls $B_k^{h_n} \coloneqq \overline{\mathscr{B}}(b_k, \tau_n s_n + h_n)$ are disjoint. We distinguish the contributions of each $B_k^{h_n}$ by using the convexity of the squared Wasserstein distance in the Banach sense; denoting $\mu_{A,k} \coloneqq \mu \resmes (A \cap B_k)$, and $\eta_{A,k} \coloneqq \eta \resmes ((A \cap B_k) \times \Sp)$, there holds
	\begin{align*}
		W^2(\pi_{y \#} \eta_A, (\mu_A)_{h_n})
		\leqslant \sum_{k=1}^K W^2(\pi_{y \#} \eta_{A,k}, (\mu_{A,k})_{h_n}).
	\end{align*}
	Now, both measures $\pi_{y \#} \eta_{A,k}$ and $(\mu_{A,k})_{h_n}$ are submeasures of $\exp_{\mu}(h_n \cdot \zeta) \resmes B_k^{h_n}$. Indeed, $\pi_{y \#} \eta = \exp_{\mu}(h_n \cdot \zeta)$, and the $\eta-$a.e. estimate on $|y-x|$ shows that $\pi_{y \#} \eta_{A,k}$ is supported on $B_k^{h_n}$. On the other hand, $\zeta$ coincides with the unit symmetric measure field on $A$. To apply \Cref{res:smallestim}, we only need to estimate the difference between $\pi_{y \#} \eta_{A,k}(\Sp) = (\mu_{A,k})_{h_n}(\Sp) = \mu(A \cap B_k)$ and the mass that $\exp_{\mu}(h_n \cdot \zeta)$ puts on $B_k^{h_n}$. But
	\begin{align*}
		\int_{(x,v)} \ind_{B_k^{h_n}}(x + h_n v) d\zeta 
		&= \int_{(x,v), x \in A} \frac{1}{2} \left[\ind_{B_k^{h_n}}(x - h_n) + \ind_{B_k^{h_n}}(x + h_n)\right] d\mu + \int_{x \in A^c} \ind_{B_k^{h_n}}(x) d\mu \\
		&= \frac{1}{2} \left[\mu(A \cap (B_k^{h_n} + h_n)) + \mu(A \cap (B_k^{h_n} - h_n))\right] + \mu\left(A^c \cap B_k^{h_n}\right) \\
		&= \mu(A \cap B_k) + \frac{1}{2} \mu\left(A \cap (B_k^{2 h_n} \setminus B_k)\right) + \mu\left(A^c \cap B_k^{h_n}\right).
	\end{align*}
	Since $A \subset \bigcup_{k=1}^K$, and $B_k,B_{\ell}$ are separated by at least $s_n - \tau_n s_n = 2 h_n + \varepsilon_n > 2 h_n$, the sets $A \cap (B_k^{2 h_n} \setminus B_k)$ are empty. On the other hand, $B_k^{h_n} \subset (A \cap B_k)^{\tau_n s_n + h_n}$ by definition. Therefore, the difference of masses is inferior to $\mu((A \cap B_k)^{\tau_n s_n + h_n} \setminus A)$. Summing over $k \in \llbracket1,K\rrbracket$, and applying \Cref{res:smallestim}, 
	\begin{align*}
		W^2(\pi_{y \#} \eta_A, (\mu_A)_{h_n})
		\leqslant \sum_{k=1}^K \mu((A \cap B_k)^{\tau_n s_n + h_n} \setminus A) \left(\diam B_k^{h_n}\right)^2
		\leqslant 4 (\tau_n s_n + h_n)^2 \mu(A^{\tau_n s_n + h_n} \setminus A).
	\end{align*}
	Note that $\tau_n s_n = \frac{\tau_n}{1 - \tau_n} (2 h_n + \varepsilon_n) \leqslant c_n h_n$ from some bounded sequence $(c_n)_n$. Hence, we obtain from \Cref{fromAtolim:secondtri} that 
	\begin{align*}
		W^2(\mu,\exp_{\mu}(h_n \cdot \zeta))
		\geqslant W^2(\mu_A, (\mu_A)_{h_n}) - 8 c_n h_n^2 \sqrt{\mu\left(A^{c_n h_n} \setminus A\right)}.
	\end{align*}
	As $A$ is compact, there holds $\bigcap_{n \in \mathbb{N}} A^{c_n h_n} \setminus A = \emptyset$, hence $\mu\left(A^{c_n h_n} \setminus A\right) = O(h_n)$. To conclude, we only have to show that $W^2(\mu_A, (\mu_A)_{h_n}) / h_n^2 \to_n \mu(A)$; this is almost the same computation as in \Cref{res:charLS1}, up to the choice of $h_n = \frac{s_n - \tau_n s_n - \varepsilon_n}{2}$ instead of $\frac{s_n + \tau_n s_n}{2}$, and the perturbation by $(2 \tau_n s_n + \varepsilon_n)/2 = o(h_n)$ does not affect the result. 
\end{proof}

\section{Measures for which the limit sup vanishes}\label{sec:LS0}

\subsection{Regularity of the distribution function}

The situation for measures $\mu$ such that $W_1(\mu,\mu_h) = o(h)$ is much less clear. All absolutely continuous measures have this property, and \Cref{res:limAlpha0} below exhibits some non-absolutely continuous measure that also does. Of course, measures that give mass to some $A \in \mathscr{A}$ do not. In between, we are only able to say the following. 

\begin{theorem}\label{res:charLS0}
	Let $\mu \in \Psp(\Sp)$, and denote $F_{\mu} : \mathbb{R} \to [0,1]$ the function $F_{\mu}(x) \coloneqq \mu((-\infty,x])$. It is equivalent that
	\begin{enumerate}
	\item\label{item:charLS0:lim} $W_1(\mu,\mu_h) = o(h)$;
	\item\label{item:charLS0:dom} for any vanishing sequence $(h_n)_{n \in \mathbb{N}} \subset (0,1]$, there exists a subsequence $(h_{n_k})_{k \in \mathbb{N}}$ and $g \in L^1(\Sp;\mathbb{R}^+)$ such that 
	\begin{align}\label{charLS0:dom:statement}
		\left|\frac{\frac{1}{2} \left[F_{\mu}(x-h_{n_k}) + F_{\mu}(x + h_{n_k})\right] - F_{\mu}(x)}{h_{n_k}}\right| \leqslant g(x)
		\qquad \text{for almost every } x \in \mathbb{R} \text{ and any } k \in \mathbb{N}.
	\end{align}
	\end{enumerate}
\end{theorem}

By \Cref{rem:limindepp}, \Cref{item:charLS0:lim} implies that $W_p(\mu,\mu_h) = o(h)$ for any $p \in [1,\infty)$. \Cref{item:charLS0:dom} can be understood as a regularity property in the following sense: the function $F_{\mu}$ is nondecreasing, so differentiable almost everywhere. Therefore, the left hand-side of \Cref{charLS0:dom:statement} converges to 0 for $\mathcal{L}-$almost every $x$. However, \Cref{charLS0:dom:statement} additionally imposes that the error is dominated by an $L^1$ function $g$.

\begin{proof}
	 By~\cite[Proposition~2.17]{santambrogioOptimalTransportApplied2015}, the 1-dimensional 1-Wasserstein distance is equivalently computed as
	\begin{align*}
		\frac{W_1(\mu,\mu_h)}{h}
		= \int_{x \in \Sp} \frac{1}{h} \left|F_{\mu_h}(x) - F_{\mu}(x)\right| dx
		= \int_{x \in \Sp} \frac{1}{h} \left|\frac{F_{\mu}(x-h) + F_{\mu}(x+h)}{2} - F_{\mu}(x)\right| dx.
	\end{align*}
	Assume \Cref{item:charLS0:lim} holds, and let $(h_n)_{n \in \mathbb{N}} \subset (0,1]$ be a vanishing sequence. Choose a subsequence $(n_k)_{k \in \mathbb{N}}$ such that $W_1(\mu,\mu_{h_{n_k}}) \leqslant h_{n_k} 2^{-k}$, and let $g \coloneqq \sum_{k} \left|\frac{1}{2} \left[F_{\mu}(\cdot -h_{n_k}) + F_{\mu}(\cdot +h_{n_k})\right] - F_{\mu}\right| / h_{n_k}$. Then $g \in L^1(\Sp;\mathbb{R}^+)$, and \Cref{item:charLS0:dom} holds. 
	
	Conversely, assume \Cref{item:charLS0:dom}. Since $F_{\mu}$ is nondecreasing and bounded, it has bounded variation, hence is differentiable Lebesgue-almost everywhere \cite[Corollary~3.29]{ambrosioFunctionsBoundedVariation2000}. Consequently, for $dx-$almost every $x \in \Sp$, there holds $\psi_h(x) \coloneqq \left|\frac{1}{2} \left[F_{\mu}(x-h) + F_{\mu}(x+h)\right] - F_{\mu}(x)\right|/h \to 0$ as $h \searrow 0$. Let $(h_n)_{n \in \mathbb{N}}$ be a vanishing sequence realizing the limit sup of $W_1(\mu,\mu_h)/h$ when $h$ goes to 0. By assumption, along some subsequence $(n_k)_k$, the functions $\psi_{h_{n_k}}$ are uniformly bounded by an $L^1$ function. Applying Lebesgue's dominated convergence, we conclude that $\lim_{k \to \infty} \|\psi_{h_{n_k}}\|_{L^1} = \limsup_{h \searrow 0} W_1(\mu,\mu_h)/h = 0$, which proves \Cref{item:charLS0:lim}. 	
\end{proof}


\subsection{Classification of uniform Cantor measures}

Instead of addressing the full characterization of measures $\mu$ such that $W(\mu,\mu_h) = o(h)$, we reduce to a subclass of uniform measures on Cantor sets. To avoid irrelevant complications, we only consider nonincreasing sequences $(\alpha_n)_n \subset (0,1)$, which therefore admit a limit in $[0,1]$. If $I$ is a bounded interval of $\mathbb{R}$, the notation $\alpha * I$ refers to the interval with the same center as $I$, and length equal to $\alpha$ times that of $I$.

\begin{definition}[Cantor set $C(\alpha)$]\label{def:Calpha}
	Let $\alpha = (\alpha_n)_{n \in \mathbb{N}} \subset (0,1)$ be a nonincreasing sequence. Denote $C^0 = I^0_0 \coloneqq [0,1]$. Assuming that $C^n$ is defined as the union of $2^n$ disjoint intervals $(I^n_{\ell})_{\ell = 1}^{2^n}$, each with the same length $\delta_n$, define $C^{n+1}$ by deleting the proportion $\alpha_n$ of the middle of each $I^n_{\ell}$:
	\begin{align*}
		C^{n+1} \coloneqq \bigcup_{\ell \in \llbracket1,2^n\rrbracket} I^n_{\ell} \setminus (\alpha_n * I^n_{\ell}).
	\end{align*}
	Then $C^{n+1}$ is again given by $2^{n+1}$ intervals. The set $C(\alpha)$ is defined as the intersection of all $(C^n)_{n \in \mathbb{N}}$.
\end{definition}

\begin{definition}[Uniform measure on $C(\alpha)$]\label{def:mualpha}
	Let $\alpha = (\alpha_n)_{n \in \mathbb{N}}$, and consider the notations of~\Cref{def:Calpha}. For each $n \in \mathbb{N}$, let $\mu^{n} \coloneqq \left(\mathcal{L} \resmes C^n\right) / \mathcal{L}(C^n)$. The measure $\mu_{\alpha}$ is defined as the $W-$limit of $(\mu^n)_{n \in \mathbb{N}}$. Equivalently, 
	\begin{align}\label{eqmu}
		\mu_{\alpha}(I^n_{\ell}) = 2^{-n}
		\qquad \forall n \in \mathbb{N} \text{ and } \ell \in \llbracket1,2^{n}\rrbracket.
	\end{align}
\end{definition}

For instance, the classical triadic Cantor set is obtained with the constant sequence $\alpha_n \equiv 1/3$, and $\mu_{\alpha}$ is the derivative of the Cantor staircase.
We are interested in the sequences $\alpha$ such that 
\begin{align}\label{LS0}
	\limsup_{h \searrow 0} \frac{W(\mu_{\alpha},\mu_{\alpha,h})}{h} = 0. 
\end{align}

The sets $C(\alpha)$ were studied in \cite{humkeMeasuresWhichSPorous1985}. In their terminology, a set $C$ is Tkadlec if there exists a measure concentrated on $C$ that gives mass to some porous set. They show that $C(\alpha)$ is Tkadlec if and only if $\liminf_{n \to \infty} \alpha_n > 0$. What we show in \Cref{res:limAlphaNot0,res:limAlpha0} can be reformulated as ``$\mu_{\alpha}$ satisfies \Cref{LS0} if and only if $C(\alpha)$ is non-Tkadlec''. However, a quick computation yields that $\mathcal{L}(C(\alpha)) > 0$ if and only if $\sum_{n \in \mathbb{N}} \alpha_n < \infty$. Therefore, any vanishing sequence $\alpha$ that is not summable generates a set $C(\alpha)$ of null Lebesgue measure, but such that $\mu_{\alpha}$ still satisfies \Cref{LS0}.

\begin{prop}\label{res:limAlphaNot0}
	Assume that \,$\lim_{n \to \infty} \alpha_n \eqqcolon \overline{\alpha} > 0$. Then \Cref{LS0} fails. 
\end{prop}

\begin{proof}
	By \Cref{res:boundR}, the property \Cref{LS0} is independent of the first terms of the sequence $(\alpha_n)_n$, and we might consider that $\alpha_n \geqslant \overline{\alpha}/2$ for any $n$.
	Denote again $\delta_n$ the length of any interval $I^{n}_{\ell}$. Let $h_n \coloneqq \alpha_n \delta_n / 2$ and $\psi : x \mapsto \inf_{y \in C^n} d(x,y)$. As $C(\alpha) \subset C^n$, there holds $\psi(x) = 0$ for $\mu_{\alpha}-$a.e. $x \in \Sp$. On the other hand, since $\psi$ is 1-Lipschitz, it is admissible in the dual problem for the 1-Wasserstein distance, and there holds
	\begin{align}\label{Cantorcomputation:lb}
		\frac{W_1(\mu_{\alpha}, \mu_{\alpha,h_n})}{h_n} 
		&\geqslant \frac{1}{h_n} \left[\int_{y \in \Sp} \psi(y) d\mu_{\alpha,h_n} - \int_{x \in \Sp} \psi(x) d\mu_{\alpha}\right]
		= \int_{x \in \Sp} \frac{\psi(x-h_n) + \psi(x+h_n)}{2 h_n} d\mu_{\alpha} \notag \\
		&= \sum_{\ell = 1}^{2^{n-1}} \int_{x \in I^{n-1}_{\ell}} \frac{\psi(x-h_n) + \psi(x+h_n)}{2 h_n} d\mu_{\alpha}
		= 2^{n-1} \int_{x \in I^{n-1}_{1}} \frac{\psi(x-h_n) + \psi(x+h_n)}{2 h_n} d\mu_{\alpha}.
	\end{align}
	In the last equality, we used the fact that each $I^{n-1}_{\ell}$ is a translation of $I^{n-1}_1$, and the corresponding self-similarity of $\mu_{\alpha}$. 
	The intersection $C^n \cap I^{n-1}_{1}$ is the union of two intervals $I_L$ and $I_{R}$, respectively the left and right connected components of $I^{n-1}_{1} \setminus (\alpha_n * I^n_{1})$. 
	
	For any $x \in [\max I_L - h_n/2, \max I_L]$, the point $x + h_n$ is at distance at least $h_n/2$ from any point of $I_L$. On the other hand, $\min I_R - \max I_L = 2 h_n$, so that $x + h_n$ is at distance at least $h_n$ of $I_R$. Therefore $\psi(x + h_n) = \inf_{z \in C_n} |x - z| \geqslant h_n/2$. Repeating the symmetric reasoning on $I_R$ and using that $\psi \geqslant 0$, we get the lower bound 
	\begin{align*}
		\int_{x \in I^{n-1}_{1}} \frac{\psi(x-h_n) + \psi(x+h_n)}{2 h_n} d\mu_{\alpha}
		\geqslant \frac{\mu([\max I_L - h_n/2, \max I_L]) + \mu([\min I_R - h_n/2, \min I_R])}{4}.
	\end{align*}
	To conclude, we estimate $\mu_{\alpha}([\max I_L - h_n/2, \max I_L])$ from below. Let $M$ be large enough so that $2^{-M} \leqslant \overline{\alpha}/8$. We consider the rightmost interval $J_m$ of $C^m \cap I_L$ as $m$ grows: for $m = n+1$, the interval $J_m$ is contained in $[\max I_L - \delta_n/2, \max I_L]$. In general, $J_m \subset [\max I_L - 2^{-(m-n)} \delta_n, \max I_L]$. When $m = n + M$, we obtain
	\begin{align*}
		2^{-(m-n)} \delta_n 
		= 2^{-M} \delta_n 
		\leqslant \frac{\overline{\alpha}}{8} \delta_n 
		\leqslant \frac{\alpha_n}{4} \delta_n
		= \frac{h_n}{2},
	\end{align*}
	which implies that $J_{n+M} \subset [\max I_L - h_n/2, \max I_L]$. Hence $\mu_{\alpha}([\max I_L - h_n/2, \max I_L]) \geqslant \mu_{\alpha}(J_{n+M}) = 2^{-(n+M)}$ by \Cref{eqmu}. Repeating the argument on $I_R$, we conclude that the estimate of \Cref{Cantorcomputation:lb} can be pursued into a lower bound by $2^{-M-2}$, which concludes the proof. 
\end{proof}

\begin{prop}\label{res:limAlpha0}
	Assume that \,$\lim_{n \to \infty} \alpha_n = 0$. Then \Cref{LS0} holds. 
\end{prop}

\begin{proof}
	Again by \Cref{res:boundR}, we can discard the first terms of the sequence and assume that $\alpha_n \leqslant 1/2$ for any $n \in \mathbb{N}$. By \cite[Proposition~2.17]{santambrogioOptimalTransportApplied2015}, the 1-dimensional 1-Wasserstein distance can be computed by 
	\begin{align*}
		W_1(\mu_{\alpha},\mu_{\alpha,h})
		= \int_{x \in \Sp} \left|F_{\mu_{\alpha,h}}(x) - F_{\mu_{\alpha}}(x)\right| dx,
	\end{align*}
	with $F_{\mu_{\alpha}}$ and $F_{\mu_{\alpha,h}}$ the respective distribution functions of $\mu_{\alpha}$ and $\mu_{\alpha,h}$. In our case, there holds in addition that $F_{\mu_{\alpha,h}}(x) = \frac{1}{2} \left[F_{\mu_{\alpha}}(x-h) + F_{\mu_{\alpha}}(x+h)\right]$. We can further represent $F_{\mu_{\alpha}}$ as a series by mimicking the construction of $\mu_{\alpha}$ as a limit of $(\mu^n)_{n \in \mathbb{N}}$: let $f_{-1}(x) \coloneqq F_{\mu^0}(x) = \max(0,\min(1,x))$, and for each $n$, let $f_n \coloneqq F_{\mu^{n+1}} - F_{\mu^n}$. Then $f_n = 0$ out of $C^n$, and $f_n$ repeats the same pattern on each interval of $C^n$. Precisely, if $[a,a+\delta_n]$ and $[b,b+\delta_n]$ are two intervals of $C^n$, then $f_n(a+z) = f_n(b+z) = f_n(z)$ for any $z \in [0,\delta_n]$, with
	\begin{align}\label{CantorTo0:explicitfn}
		f_n(z) = 
		\begin{cases}
			2^{-(n+1)} \frac{z}{\delta_{n+1}} - 2^{-n} \frac{z}{\delta_n} & z \in [0,\delta_{n+1}], \\
			2^{-(n+1)} - 2^{-n} \frac{z}{\delta_n} & z \in (\delta_{n+1},\delta_n - \delta_{n+1}), \\
			2^{-(n+1)} + 2^{-(n+1)} \frac{z - (\delta_n - \delta_{n+1})}{\delta_{n+1}} - 2^{-n} \frac{z}{\delta_n} & z \in [\delta_n - \delta_{n+1}, \delta_{n}].
		\end{cases}		 
	\end{align}
	Recall that $\delta_{n+1} = \delta_n (1-\alpha_n)/2$. One checks that $\|f_n\|_{\infty} = f_n(\delta_{n+1}) = 2^{-(n+1)} \alpha_n$.
	Hence $\sum_{n = -1}^{\infty} f_n$ converges uniformly towards $F_{\mu_{\alpha}}$. We now bound each term of the series.
	
	For each $h$, let $n \in \mathbb{N}$ be such that $h \in [\sqrt{\alpha_n} \delta_{n+1}, \sqrt{\alpha_{n-1}} \delta_{n})$. We have to estimate 
	\begin{align*}
		\int_{x \in \Sp} \frac{\left|\frac{1}{2} \left[F_{\mu_{\alpha}}(x-h) + F_{\mu_{\alpha}}(x+h)\right] - F_{\mu_{\alpha}}(x)\right|}{h} dx
		\leqslant \sum_{k = -1}^{\infty} \int_{x \in \Sp} \frac{\left|\frac{1}{2} \left[f_k(x-h) + f_k(x+h)\right] - f_k(x)\right|}{h} dx.
	\end{align*}
	 We treat differently the indices $k < n$ and $k \geqslant n$. For any $k \geqslant n$, the application $x \mapsto \frac{1}{2} \left[f_k(x-h) + f_k(x+h)\right] - f_k(x)$ vanishes for any point $x$ at distance at least $h$ from $C^{n}$. Hence
	\begin{gather*}
		\sum_{k = n}^{\infty} \int_{x \in \Sp} \frac{\left|\frac{1}{2} \left[f_k(x-h) + f_k(x+h)\right] - f_k(x)\right|}{h} dx
		\leqslant 2 \sum_{k = n}^{\infty} \frac{\|f_k\|_{\infty}}{h} \mathcal{L}\left(\left\{ x \in \mathbb{R} \st d(x,C^{n}) \leqslant h \right\}\right) \\
		\leqslant 2 \sum_{k = n}^{\infty} \frac{2^{-(k+1)} \alpha_k}{h} 2^{n} \left(\mathcal{L}(I^{n}_1) + 2 h\right) 
		= \sum_{k = n}^{\infty} \alpha_k 2^{n-k} \left(\frac{\delta_{n+1}}{h} + 2\right)
		\leqslant \sum_{k = n}^{\infty} \alpha_k 2^{n-k} \left(\frac{1}{\sqrt{\alpha_n}} + 2\right)
		\leqslant \sqrt{\alpha_n} + 2 \alpha_n.
	\end{gather*}
	Here we used that $\alpha_k \leqslant \alpha_n$ for $k \geqslant n$.
	Now, for $k < n$, the application $x \mapsto \frac{1}{2} \left[f_k(x-h) + f_k(x+h)\right] - f_k(x)$ differs from 0 only around the points of discontinuity of $f_k'$. For $\alpha_n \leqslant 1/2$, the Lipschitz constant of $f_k$ can be seen in \Cref{CantorTo0:explicitfn} to be the slope of the middle segment, so that $\left|\frac{1}{2} \left[f_k(x-h) + f_k(x+h)\right] - f_k(x)\right| \leqslant 2^{-k} h/\delta_k$. Now, counting $2^k$ intervals of $C^k$, each containing 4 points of discontinuity, each influencing a segment of length $2h$, we can estimate that
	\begin{align*}
		\int_{x \in \mathbb{R}} \frac{\left|\frac{1}{2} \left[f_k(x-h) + f_k(x+h)\right] - f_k(x)\right|}{h} dx
		\leqslant \frac{1}{h} \times 2^{k} \times 4 \times 2 h \times \frac{2^{-k} h}{\delta_{k}}
		\leqslant \frac{8 h}{\delta_{k}}
		\leqslant \frac{8 h}{2^{n-k} \delta_{n}}
		\leqslant 2^{k-n+3} \sqrt{\alpha_n}.
	\end{align*}
	In the last equalities, we used that $\delta_{n+1} \leqslant \delta_n/2$, hence $\delta_k \geqslant 2^{n-k} \delta_n$, and the choice $h \leqslant \sqrt{\alpha_n} \delta_n$. In consequence, 
	\begin{align*}
		\sum_{k=0}^n \int_{x \in \mathbb{R}} \frac{\left|\frac{1}{2} \left[f_k(x-h) + f_k(x+h)\right] - f_k(x)\right|}{h} dx
		\leqslant \sum_{k=0}^n 2^{k-n+3} \sqrt{\alpha_n}
		\leqslant 16 \sqrt{\alpha_n}.
	\end{align*}
	There only stays to bound the contribution of $k = -1$, which is easily seen to be of order $4 h$. Gathering the above estimates, we conclude that $W_1(\mu_{\alpha},\mu_{\alpha,h})/h \to 0$ when $h \searrow 0$, hence that \Cref{LS0} holds. 
\end{proof}

\paragraph{Acknowledgments} This work benefited from the support of the ERC Starting Grant ConFine n°101078057 in Pisa.

\end{document}